\documentclass[11pt]{article}
\usepackage{amssymb,amsmath,float}
\usepackage{amsthm}  

\usepackage[total={6.5in,8.75in}, top=1.2in, left=0.9in, includefoot]{geometry}
\usepackage{graphicx}
\usepackage{tikz}
\usepackage{enumitem}
 \usepackage[T1]{fontenc}    

\usepackage{url}
\usepackage{xcolor}

\newcommand{\Eq}[1]{(\ref{eq:#1})}

\newcommand{\Lem}[1]{Lem.~\ref{lem:#1}}

\newcommand{\Sec}[1]{\S \ref{sec:#1}}
\newcommand{\Fig}[1]{Fig.~\ref{fig:#1}}
\newcommand{\Tbl}[1]{Table~\ref{tbl:#1}}
\newcommand{\App}[1]{Appendix~\ref{app:#1}}

\newcommand{\InsertFig}[4]
{\begin{figure}[h!t]
       \centerline{
         \includegraphics[width=#4\columnwidth]{./figures/#1}
       }
       \caption{{\footnotesize  #2}
       \label{fig:#3}}
\end{figure}}

\newcommand{\InsertFigTwo}[5] {
\begin{figure}[h!t]
       \centerline{
         \includegraphics[width=#5\columnwidth]{./figures/#1}
         \hskip 0.5in
         \includegraphics[width=#5\columnwidth]{./figures/#2}
       }
       \caption{{\footnotesize  #3}
       \label{fig:#4}}
\end{figure}}
\newcommand{\InsertFigThree}[6] {
\begin{figure}[ht]
       \centerline{
         \includegraphics[width=#6\columnwidth]{./figures/#1}
         \hskip 0.5in
         \includegraphics[width=#6\columnwidth]{./figures/#2}
         \hskip 0.5in
         \includegraphics[width=#6\columnwidth]{./figures/#3}
       }
       \caption{{\footnotesize  #4}
       \label{fig:#5}}
\end{figure}}
\newcommand{\InsertFigFour}[7] {
\begin{figure}[ht]
       \centerline{
\renewcommand{\arraystretch}{0.01}
         \begin{tabular}{cc}
         \includegraphics[width=#7\columnwidth]{./figures/#1}&  \includegraphics[width=#7\columnwidth]{./figures/#2} \\
        \includegraphics[width=#7\columnwidth]{./figures/#3}  &  \includegraphics[width=#7\columnwidth]{./figures/#4}
        \end{tabular}
       }
       \caption{{\footnotesize  #5}
       \label{fig:#6}}
\end{figure}}

\newcommand{\bN}{{\mathbb{ N}}}

\newcommand{\bR}{{\mathbb{ R}}}

\newcommand{\bZ}{{\mathbb{ Z}}}

\newcommand{\cB}{{\cal B}}
\newcommand{\cE}{{\cal E}}
\newcommand{\cF}{{\cal F}}
\newcommand{\cG}{{\cal G}}
\newcommand{\cL}{{\cal L}}
\newcommand{\cO}{{\cal O}}
\newcommand{\cR}{{\cal R}}

\newcommand{\eps}{\varepsilon}

\newcommand{\hen} {H\'enon }
\DeclareMathOperator{\sign}{sign}

\newcommand{\Tvec}[2]{\begin{array}{c} {#1}\\{#2}\end{array}}
\newcommand{\Thvec}[3]{\begin{array}{c} {#1}\\{#2}\\{#3}\end{array}}

\newtheorem{thm}{Theorem}
\newtheorem{lem}[thm]{Lemma}

\newtheorem{con}[thm]{Conjecture}

\newcommand{\beq}[1]{\begin{equation}\label{eq:#1}}
\newcommand{\eeq}{\end{equation}}

\newenvironment{se}[1]{\equation\label{eq:#1}\aligned}{\endaligned\endequation}
\newcommand{\bsplit}[1]{\begin{se}{#1}}
\newcommand{\esplit}{\end{se}}

\newenvironment{example}[1][]
  {
	\setlength \leftmargini {1.0em}		
	\setlength \topsep {0.5em}			
	\begin{quote}
	{\it Example#1} }
	{\end{quote}
  }
\newcommand{\bexam}[1][:]{\begin{example}[#1]}
\newcommand{\eexam}{\end{example}}

\title{Anti-Integrability for Three-Dimensional Quadratic Maps}
\author{Amanda E. Hampton and James D.~Meiss\thanks
      {
        The authors were supported in part by NSF grant DMS-181248.
        Useful conversations with Holger Dullin are gratefully acknowledged. 
      }
    \\
 \begin{tabular}{c}
	Department of Applied Mathematics\\
    University of Colorado \\
	Boulder, CO 80309-0526 \\
	Amanda.Hampton@colorado.edu\quad
	James.Meiss@colorado.edu\\ 
\end{tabular}
}

\date{\today}

\begin{document}
\maketitle

\begin{abstract}

We study the dynamics of the three-dimensional quadratic diffeomorphism using a concept first introduced thirty years ago for the Frenkel-Kontorova model of condensed matter physics: the anti-integrable (AI) limit. At the traditional AI limit, orbits of a map degenerate to sequences of symbols and the dynamics is reduced to the shift operator, a pure form of chaos. Under nondegeneracy conditions, a contraction mapping argument can show that infinitely many AI states continue to orbits of the deterministic map. For the 3D quadratic map, the AI limit that we study is a \textit{quadratic correspondence} whose branches, a pair of one-dimensional maps, introduce symbolic dynamics on two symbols. The AI states, however, are nontrivial orbits of this correspondence. The character of these orbits depends on whether the quadratic takes the form of an ellipse, a hyperbola, or a pair of lines. Using contraction arguments, we find parameter domains for each case such that each symbol sequence corresponds to a unique AI state. In some parameter domains, sufficient conditions are then found for each such AI state to continue away from the limit to become an orbit of the original 3D map. Numerical continuation methods extend these results, allowing computation of bifurcations and obtaining orbits with horseshoe-like structures and intriguing self-similarity.  We conjecture that pairs of periodic orbits in saddle-node or period doubling bifurcations have symbol sequences that differ in exactly one position.
\end{abstract}

\section{Introduction}\label{sec:Intro}

In this paper we  explore the chaotic behavior of three-dimensional quadratic diffeomorphisms by constructing a limit in which the dynamics is simple, though essentially fully chaotic: the anti-integrable (AI) limit. As shown in \cite{Lomeli98, Lenz99}, the general quadratic diffeomorphism $L:\bR^3 \to \bR^3$ with quadratic inverse can be written in the form 
\bsplit{QVPMap}
	L(x,y,z) &= (\delta z+ G(x,y), x,y), \\
	G(x,y)   &= \alpha + \tau x - \sigma y + ax^2 + bxy + cy^2 .
\esplit
Here $G$ is an arbitrary quadratic polynomial with the six parameters $(\alpha,\tau,\sigma,a,b,c)$ and $\delta = \det{DL}$, the Jacobian determinant of the map.
This map is also the quadratic normal form near a saddle-center bifurcation with a triple-one multiplier \cite{Dullin08a}. The orbits of \Eq{QVPMap} are sequences $\{(x_t,y_t,z_t): t \in \bZ\}$ that satisfy
\[
   (x_{t+1},y_{t+1},z_{t+1}) = L(x_{t},y_{t},z_{t}).
\]

The volume-preserving case, $\delta = 1$, is of interest as an idealized model for the motion of particles in an incompressible fluid. Studying the dynamics of such maps is important in understanding Lagrangian mixing problems, with many applications \cite{Haller13}.
The volume-contracting case, $|\delta| < 1$, arises as a normal form near homoclinic bifurcations of 3D maps \cite{Gonchenko06} and can give rise to discrete Lorenz-like attractors \cite{Gonchenko21}.

\textit{Anti-integrability} was first introduced for discrete-time Hamiltonian systems by Aubry and Abramovici \cite{Aubry90} in contradistinction to \textit{integrability}, where the system has a full set of invariants and the phase space is foliated by invariant tori. A perturbation of an integrable system gives dynamics that can sometimes be described by KAM theory: many tori are preserved under smoothness and nondegeneracy conditions. By contrast, anti-integrability corresponds to non-deterministic dynamics: in the simplest case AI dynamics can be described by a full-shift on a set of symbols. 

The anti-integrable limit was originally introduced for area-preserving maps through
the action, which is a sum of discrete Lagrangians or generating functions,
\[
    S_\eps(x_t,x_{t+1}) = \eps K(x_t-x_{t+1}) - V(x_t),
\]
along a sequence $\{x_t: t \in \bZ\}$ of the configuration spatial variables \cite{Aubry95,Bolotin15}. Each critical point of the action corresponds to an orbit. Physically, $K$ represents the kinetic energy (in discrete time) and $V$ the potential energy. The resulting dynamics is integrable when the potential energy vanishes. In contrast, the AI limit can be interpreted as the kinetic energy, or equivalently $\eps$, going to zero. In this case critical points of the action become arbitrary sequences of critical points of the potential energy, and therefore the ``dynamics'' is non-deterministic. When the AI limit is ``nondegenerate'' then each AI state continues to nonzero $\eps$, becoming a true orbit of the original system. This theory and methodology has proved useful for analyzing two-dimensional systems \cite{Aubry90, Aubry95, Bolotin15} as well as higher-dimensional symplectic maps \cite{MacKay92b}.

One can generalize this idea to more general maps by converting them to implicit equations on a sequence $\{x_t\}$. For example, it has been applied to the 1D logistic map \cite{Chen06} and---as we will recall below---to H\'enon's 2D map \cite{Aubry95, Sterling98, Sterling99, Dullin05, Chen06}. Results of this analysis give an alternative argument to that of Devaney and Nitecki \cite{Devaney79} for the existence of a Smale horseshoe in a domain of parameter space, which implies the existence of a set of hyperbolic orbits conjugate to a Bernoulli shift on two symbols. In addition one can find parameter sets for which there are specific sub-shifts, study the first bifurcations away from $\eps = 0$, and---for the area-preserving case---find the correspondence between symbol sequences and orbits of a given rotation number \cite{Dullin05}.

The theory has been generalized to higher-order scalar difference equations \cite{Du06,Li06,Juang08,Li10} and has been used to extend \cite{MacKay92b} to more general multi-dimensional cases \cite{Chen15}. As we will review in \Sec{3DAI}, some of these results give AI limits for special cases of three-dimensional diffeomorphisms like \Eq{QVPMap}. 

To be more specific, consider a family of maps $f: M \to M$ on a phase space $M$ so that the corresponding orbits are sequences $Z = \{z_t: t \in \bZ\}$ determined by $z_{t+1} = f(z_{t})$. The general idea of an AI limit is to rewrite the dynamics as an implicit relation $\cL_\eps(Z) = 0$ with a parameter $\eps$.
An AI limit corresponds to a singular limit of this implicit relation. For example, suppose that for $\eps = 0$, $\cL_0(Z) = 0$ degenerates to a set of equations that depends only upon individual points along the orbit, say $\cL_0(z_t) = 0$. 
If this equation has multiple solutions then the dynamics degenerates to symbolic dynamics: an element $z_t$ of the sequence can be chosen to be \textit{any} of the solutions of $\cL_0(z_t) = 0$ for each $t$. Under an implicit function style result, it can be possible to continue these symbolic orbits so that they become true orbits of the original map when $\eps \neq 0$.

As an example, the \hen map \cite{Henon76} $H: \bR^2 \to \bR^2$ can be written
\beq{HenonMap}
    (x',y')=H(x,y)=(y-k+x^2,-\delta_H x),
\eeq
with Jacobian $|DH| = \delta_H$ and parameter $k$. Instead of considering the implicit form of the map on $\bR^2$, it is more convenient to rewrite \Eq{HenonMap} as a second difference equation: let $(x_{t+1},y_{t+1})=H(x_t,y_t)$ and note that $y_{t} = -\delta_H x_{t-1}$, so that \Eq{HenonMap} is equivalent to
\beq{HenonDifference}
     x_{t+1} + \delta_H x_{t-1} = x_t^2 -k.
\eeq
The AI limit corresponds to $k\to\infty$. Defining $\eps=\frac{1}{\sqrt{k}}$ and rescaling $\xi=\eps x$, \Eq{HenonDifference} can be written as the implicit relation
\[
   \cL_\eps (\xi_{t+1},\xi_t,\xi_{t-1}) = \xi_t^2 - 1 -\eps(\xi_{t+1}+\delta_H \xi_{t-1}) = 0
\]
At the AI limit, $\eps \to 0$, this second difference equation degenerates to the relation
\[
    \xi_t^2=1.
\]
This limit is singular in the sense that the implied dynamics are no longer deterministic; indeed, orbits now correspond to any bi-infinite sequence with $\xi_t = \pm 1$ for all $t \in \bN$. Given such a sequence, the dynamics can be represented by the shift on two symbols, $\{-,+\}$.
We think of these as ``AI states'' as they are no longer ``orbits'' of a deterministic map. These AI states continue to orbits of \Eq{HenonDifference} for small enough $\eps$, as was proven using a contraction mapping argument \cite{Sterling98}. 

Our goal here is to use a similar approach for a large subset of the full seven-dimensional parameter space of the family of maps \Eq{QVPMap}. We obtain a number of possible AI limits in \Sec{3DAI} and discuss previous research that has been applied to special cases of \Eq{QVPMap}. In \Sec{Formulation} we specify the particular AI limit on which this paper focuses. A key feature of this case is that the AI limit corresponds to iteration of a quadratic correspondence, thus introducing symbolic dynamics. In \Sec{Existence}, we find parameters such that orbits exist for every symbol sequence at the AI limit. In \Sec{PersistOrbit}, we find parameters so that orbits at the AI limit can be continued away from the limit for sufficiently small $\eps>0$. Numerical methods are used in \Sec{Continuation} to continue orbits beyond the analytical bounds for several examples.

\section{Anti-Integrable Limits of 3D Quadratic Maps}\label{sec:3DAI}
Here we formulate possible AI limits of the 3D quadratic map \Eq{QVPMap}. This map can be treated as a five parameter family. Indeed, it was shown in \cite{Lomeli98} that whenever $a+b+c \neq 0$ and $2a+b \neq 0$\footnote
{If one of these is violated, other scaling transformations can be found to eliminate two of the parameters.}
an affine coordinate transformation allows one to to set
\beq{ParamSpace}
    a+b+c = 1, \mbox{ and} \quad \tau = 0 .
\eeq
We adopt this simplification so the map only depends on say $(\alpha, \sigma, a, c)$, in addition to the Jacobian $\delta$. 
To mimic the analysis for the \hen map, first rewrite \Eq{QVPMap} as a third-order difference equation for $\{x_t: t \in \bZ\}$ upon noting that $y_{t+1} = x_t$ and $z_{t+1} = x_{t-1}$,
\beq{VPDiff}
	x_{t+1}-\delta x_{t-2} = G(x_t,x_{t-1}).
\eeq
There are a number of ways to scale parameters to get an anti-integrable limit. We will assume that $a,b,c$ are ``structural'' parameters that remain finite and that $|\delta| \le 1$ so that the map is not volume expanding.
Following the \hen example in \Sec{Intro}, we introduce $\eps$ by scaling the phase space variables, defining $\xi_t = \eps x_t$. Then using \Eq{ParamSpace}, \Eq{VPDiff} becomes
\bsplit{Rescaled}
    0 & = \cL_\eps(\xi_{t+1},\xi_t,\xi_{t-1},\xi_{t-2}) \\
      &= Q( \xi_t, \xi_{t-1}) + \eps^2 \alpha - 
          \eps(\xi_{t+1} + \sigma \xi_{t-1}- \delta \xi_{t-2}),
\esplit
where we define the quadratic form
\beq{Qdefine}
	Q(x,y) \equiv ax^2 + b xy + cy^2.
\eeq
The discriminant of $Q$,
\beq{DeltaDefine}
    \Delta \equiv b^2-4ac, 
\eeq
will play a crucial role in the analysis below.

We can categorize different AI limits by scaling the remaining parameters $\alpha$ and $\sigma$ with $\eps$ so that the third-order difference equation \Eq{Rescaled} degenerates 
to a lower-order system at $\eps = 0$. There are four potential cases:
\begin{enumerate}[label={(\arabic*)}]
\item[(0)]$\alpha, \sigma$ finite: In this case the AI limit corresponds simply to setting $\eps = 0$ in \Eq{Rescaled}
to obtain $Q(\xi_t, \xi_{t-1}) = 0$. The only nontrivial case is then a pair of lines through the origin when $\Delta >0$.
\item[(1$\pm$)] $\alpha \to \pm \infty$: Define $\eps$ by $\alpha=\pm \eps^{-2}$, then \Eq{Rescaled} becomes
\beq{AIDiff}
	\eps(\xi_{t+1} + \sigma \xi_{t-1}  - \delta \xi_{t-2})  = Q( \xi_t, \xi_{t-1}) \pm 1 .
\eeq
As $\eps \to 0$, the map degenerates to $Q(\xi_t, \xi_{t-1}) = \mp 1$,
which defines a quadratic curve. The character of the possible AI states depends on the value of the discriminant \Eq{DeltaDefine}.
\item[($2\pm$)] $\sigma \to \pm \infty$: Set $\sigma=\pm \eps^{-1}$ so that \Eq{Rescaled} becomes
\beq{SigmaAILimit}
    \eps(\xi_{t+1}-\delta \xi_{t-2}-\eps\alpha)=Q(\xi_t,\xi_{t-1}) \mp \xi_{t-1},
\eeq
which yields the AI limit $Q(\xi_t,\xi_{t-1})= \pm\xi_{t-1}$, a shifted quadratic curve.
\item[($3\pm$)]$\alpha, \sigma \to \pm \infty$: Let $\alpha=\pm\eps^{-2}$, and assume that $\sigma= r \eps^{-1}$, so that $\sigma^2/\alpha = \pm r^2$ is finite, to obtain
\[
    \eps(\xi_{t+1}-\delta \xi_{t-2}) = Q(\xi_t,\xi_{t-1})  - r\xi_{t-1} \pm 1,
\]
leading to the AI limit $Q(\xi_t,\xi_{t-1})= - r\xi_{t-1}\mp1 $, again a shifted quadratic curve.
\end{enumerate}

Some of these AI limits have been studied elsewhere.
Juang et al. \cite{Juang08} consider multidimensional difference equations, $\cL_\eps(x_t,\ldots, x_{t+n}) = 0$ that reduce to $\cL_0(x_t,x_{t+k}) = 0$ when $\eps = 0$. They assume this implicit equation has a branch $x_{t+k} = f(x_t)$ so that the one-dimensional map $f$ is piecewise analytic and has an invariant set with positive topological entropy. In this case, there exists an $\eps_0 > 0$ such that these AI states become orbits of the difference equation $\cL_\eps = 0$, and depend continuously on $\eps \in [0,\eps_0)$ in the uniform topology. This gives a set of orbits of positive topological entropy for the original difference equation.

These ideas generalize those in \cite{Du06} which had applications to Arneodo-Coullet-Tresser maps and to the map \Eq{QVPMap}. For the quadratic map, this analysis corresponds to case ($2+$) above,  $\sigma \to \infty$. They suppose that $b=0$ (or more generally that $b = \cO(\eps)$)
to obtain
the AI limit associated with \Eq{SigmaAILimit}. Solving this equation for $\xi_t$ and choosing a branch gives the map
\[
	\xi_{t} =f_+(\xi_{t-1}) =  \sqrt{\frac{1}{a}(\xi_{t-1} - c\xi_{t-1}^2)} ,
\]
when $1 > a = 1-c>0$, on the domain $[0,1/c]$. This has maximum $f_+(\tfrac{1}{2c}) = \tfrac{1}{2\sqrt{ac}}$, so the map has an escaping interval if $c/a > 4$, or $c \in (\tfrac45,1)$. They note that this map has a compact invariant set with positive topological entropy when  $c > 0.791$, and consequently obtain a persistence result for the range $c \in (0.791,1)$.

A related result was obtained in \cite{Li06} where it was assumed that $\cL_0(x_t)$ depends only on one variable and has at least two simple roots. Under this assumption, the AI limit is similar to that of the \hen map discussed in \Sec{Intro}. The limit studied in \cite{Li06} corresponds to case ($1-$), $\alpha \to -\infty$, assuming that $b=c=0$ (or more generally  $b,c = \cO(\eps)$). Under these assumptions the AI limit associated with \Eq{AIDiff} reduces to
\[
    \xi_t^2 = 1 ,
\]
so that the AI states are $\{\xi_t: t \in \bZ\} \subset \{-1, 1\}^\bZ$, and the ``dynamics'' at the AI limit becomes a full shift on two symbols. This is a special case of the analysis presented in \Sec{Formulation} below.

Another case in which $\cL_0$ depends only on a single variable was studied as an application of the results in \cite{Chen16}. Their limit corresponds to case ($2-$) assuming $a=b=0$ (or more generally $a,b = \cO(\eps)$), where AI limit associated with \Eq{SigmaAILimit} simplifies to
\[
    \xi_{t-1}(\xi_{t-1} + 1) = 0 ,
\]
which again has two possible solutions and the AI states are  $\{0,-1\}^\bZ$.

\section{Quadratic Correspondences at the Anti-integrable Limit}\label{sec:Formulation}
In this paper we will consider the AI limit corresponding to case (1) of \Sec{3DAI}
where $\alpha \to \pm \infty$. 
Note that when $\alpha \to \infty$, the only nontrivial case occurs when $\Delta > 0$, since
then the set $Q = -1$ has nonzero solutions. In the alternative case, however, i.e., when $\alpha \to -\infty$, the set $Q=1$ can have nontrivial solutions for any value of $\Delta$. Thus
we will exclusively consider this latter case.

 In this case, $\alpha \to -\infty$,  we define the difference equation
\beq{OurAIDiff}
    \cL_\eps(\xi_{t+1},\xi_t,\xi_{t-1},\xi_{t-2}) = 
      Q( \xi_t, \xi_{t-1}) - 1 -\eps(\xi_{t+1} + \sigma \xi_{t-1}  - \delta \xi_{t-2}) ,
\eeq
so that $\cL_0 = 0$ corresponds to
\beq{AIRelation}
    Q(\xi_t,\xi_{t-1}) = 1.
\eeq
We generalize the study of \cite{Li06} to allow for arbitrary values of the discriminant \Eq{DeltaDefine}. In this case the ``dynamics'' at the AI limit is not deterministic, but
rather a \textit{quadratic correspondence} \cite{Bullett88,McGehee91}.

The relation \Eq{AIRelation} defines the quadratic curve 
\beq{EDelta}
    \cE = \left\{(u,v): av^2 + buv + cu^2 = 1  \right\} \subset \bR^2 .
\eeq
Sequential points on a valid AI trajectory must lie on the curve 
$(\xi_{t-1},\xi_t) \in \cE$. As sketched in \Fig{Cobwebs}, $\cE$ is either a pair of parallel lines when $\Delta=0$, an ellipse when $\Delta <0$, or a hyperbola when $\Delta > 0$. Note that $\cE$ is reflection symmetric through the origin and, since $a+b+c=1$, 
always intersects the two points $( 1, 1)$ and $(-1,-1)$.

The quadratic curve \Eq{EDelta} can be thought of dynamically as a relation or correspondence. Thus, assuming $a \neq 0$, an AI \textit{state} is defined by a sequence $\{\xi_t: t\in \bZ\}$
that satisfies
\beq{MapAsFxn}
    \xi_t = f_{s_{t}}(\xi_{t-1}) = \frac{1}{2a}\left(-b\xi_{t-1}+ 
         s_t\sqrt{\Delta \xi_{t-1}^2+4a}\right), \quad s_t \in \{-,+\} .
\eeq
Provided that $\Delta \xi_{t-1}^2 + 4a > 0$, either choice $s_t = \pm$ is valid, and the 
quadratic curve is represented by the two branches $f_\pm$ of \Eq{EDelta}.
Therefore, each valid AI state is associated with a sequence of symbols, 
\beq{SymbolSpace}
    s = \{\ldots s_0,s_1,s_2\ldots \}  \in \Sigma = \{- ,+\}^\infty ,
\eeq
so that $s_t$ represents the branch of the function implicitly defined by \Eq{AIRelation} that is chosen at time $t$.

Examples of the iteration of \Eq{MapAsFxn} for three cases are shown in \Fig{Cobwebs}.
When $s_t=-$, the point $(\xi_{t-1},\xi_t)$ lies on the lower, negative branch of the quadratic curve, and when $s_t=+$, it lies on the upper, positive branch of the curve. The red loop in the figure is the iteration corresponding to the period-four orbit, $\{-,-,+,+\}$.

\InsertFigThree{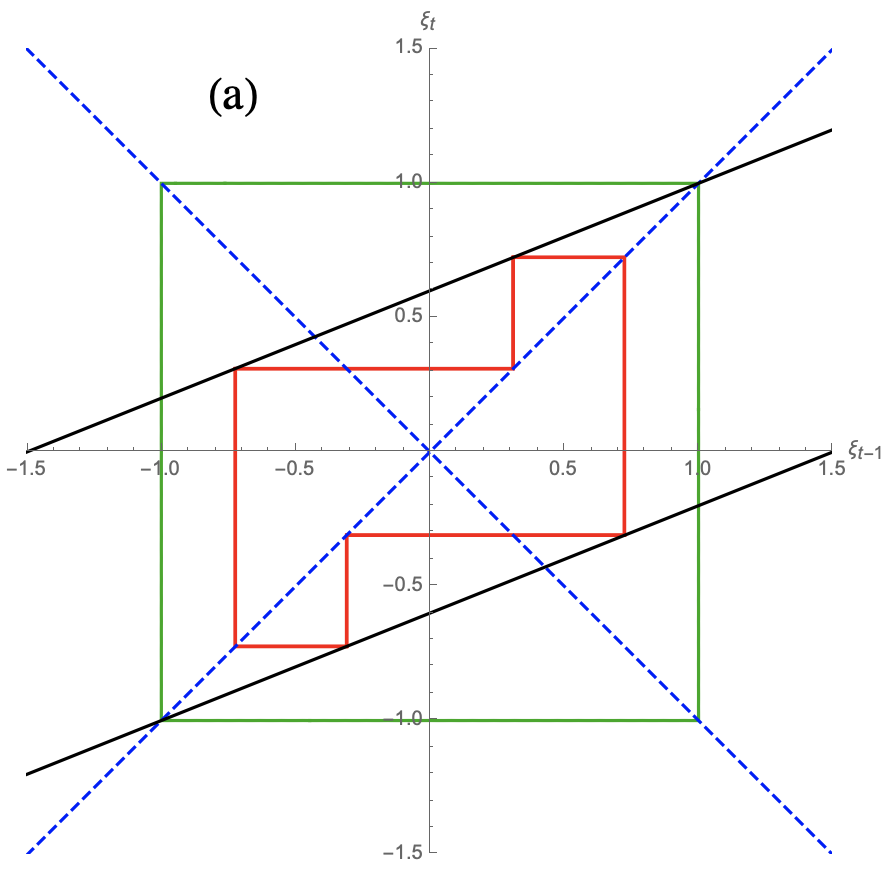}{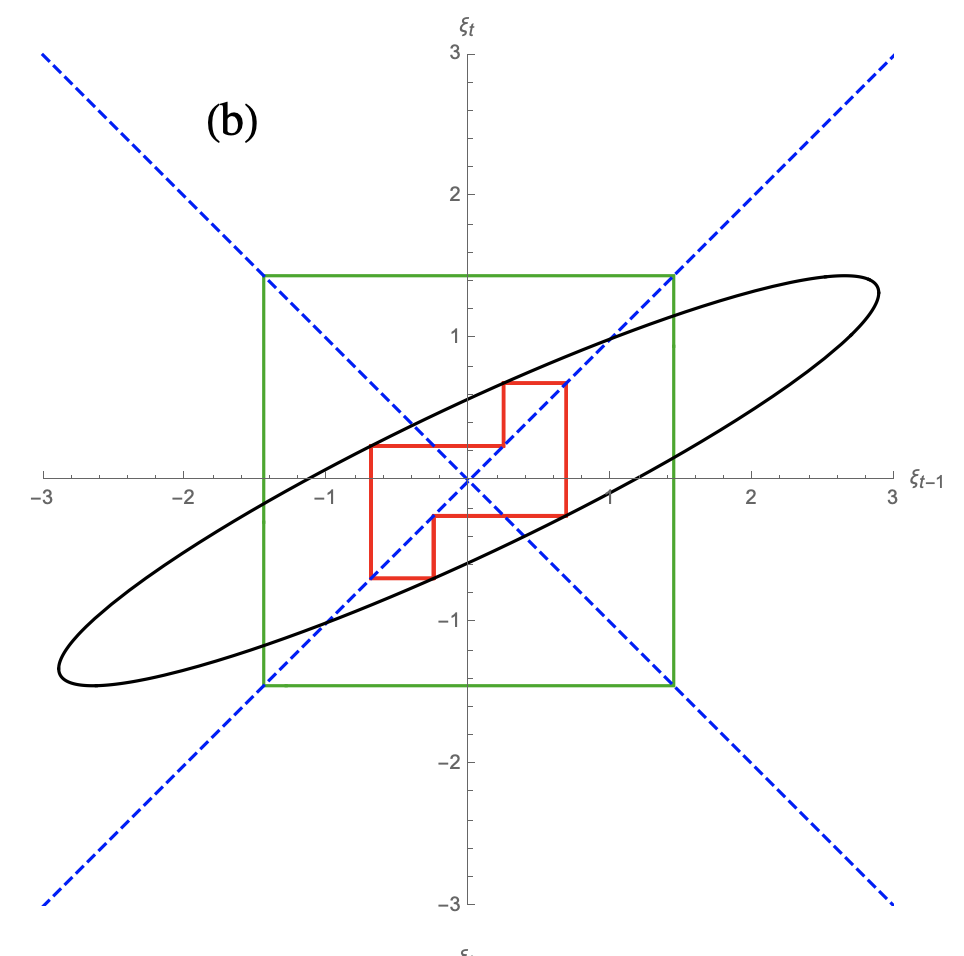}{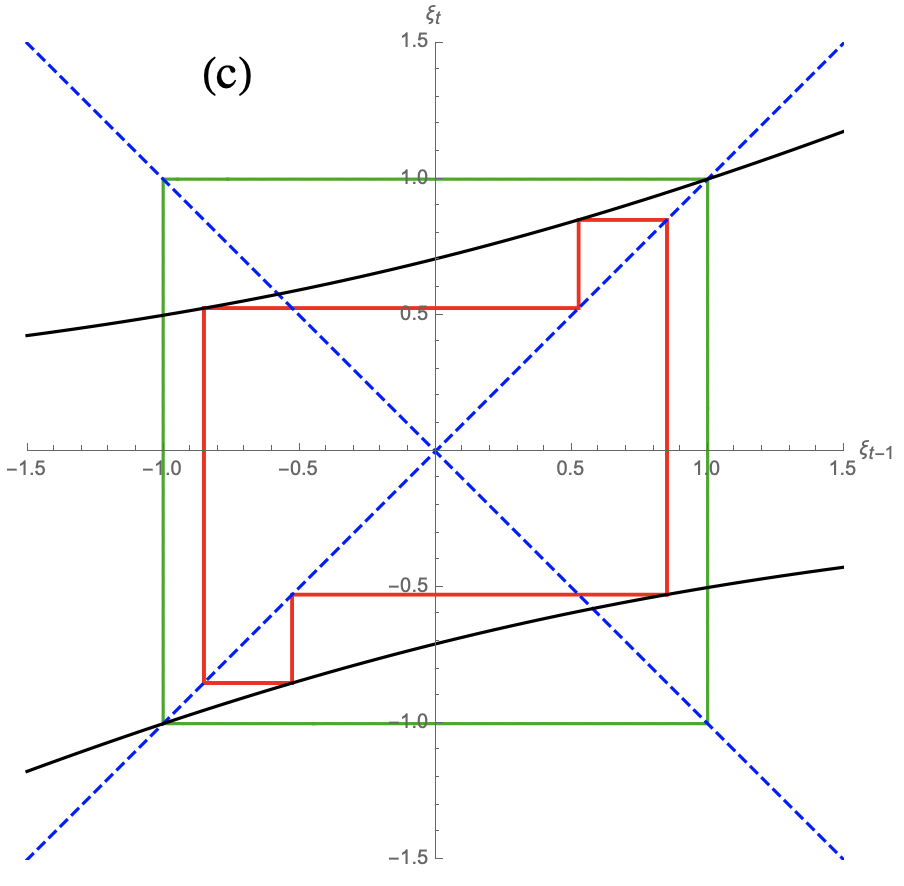}{
Three examples of the curve \Eq{EDelta} and iteration of the relation \Eq{MapAsFxn}:  
(a) parallel lines, $\Delta= 0$, with $(a,b,c) = \tfrac19(25,-20,4)$
(b) ellipse, $\Delta < 0$, with $(a,b,c)=(3,-2.75,0.75)$, and 
(c) hyperbola, $\Delta > 0$, with $(a,b,c)=(2,-1,0)$. 
The curve $\cE$ is black and the diagonals $\xi_t=\pm\xi_{t-1}$ are dashed blue. 
The red loop shows the orbit $\{-,-,+,+\}$: vertical segments from $(\xi_{t-1},\xi_{t-1})$ 
to  a branch of the curve $Q=1$ at $(\xi_{t-1},\xi_t)$ followed by a
horizontal segments to $(\xi_t,\xi_t)$.
The vertical direction depends on the symbol: the orbit moves up if $s_t = +$ and down if $s_t = -$. The interior of the green box is the set $B^2$, \Eq{BDefine}, a forward invariant set on which \Eq{MapAsFxn} is a contraction.}
{Cobwebs}{0.3}

Of course, the map represented by the third-order difference equation $\cL_\eps = 0$
is three-dimensional. Indeed \Eq{OurAIDiff} can be thought of as determining
$\xi_{t+1}$ as a function of three preceding points.
Thus the curve \Eq{EDelta} becomes a surface in $\bR^3$.
Since this relation must hold for all $t$, AI states must lie on the intersection of the
two surfaces $Q(\xi_t,\xi_{t-1})=1$ and $Q(\xi_{t+1},\xi_t)=1$,
labeling the axes by $(\xi_{t-1},\xi_{t},\xi_{t+1}) \in \bR^3$. An example is shown in \Fig{3DAISurface} for the elliptic case, when $\Delta < 0$.

\InsertFig{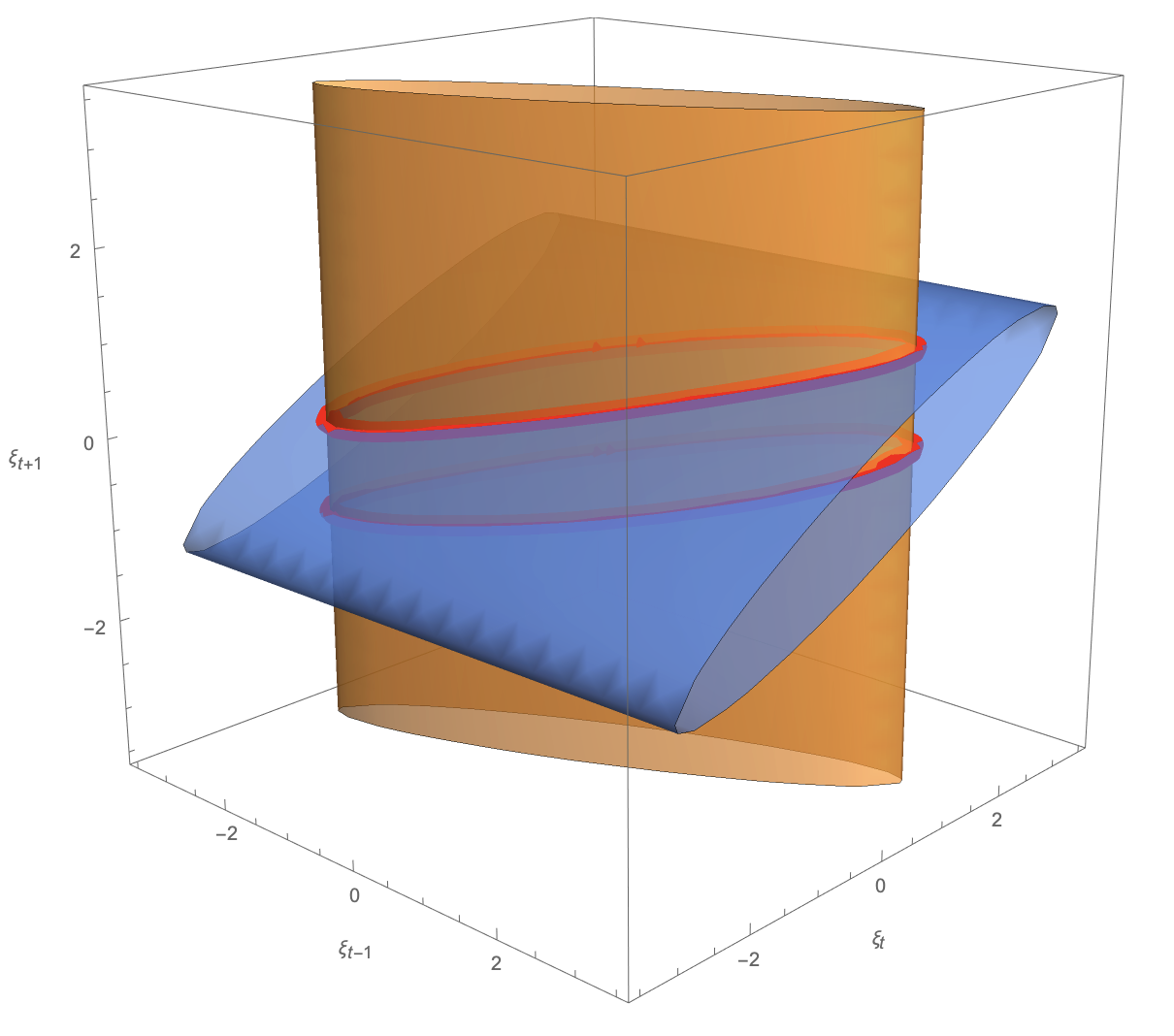}{AI states of the 3D map must lie on the intersection of 
two quadratic surfaces $Q(\xi_t,\xi_{t-1}) = 1$ (orange) and $Q(\xi_{t+1},\xi_t) = 1$ (blue). Pictured is the elliptic example of \Fig{Cobwebs}(b).
Valid AI states must lie on the intersection of these two surfaces, i.e., the two curves that are highlighted in red.}
{3DAISurface}{0.5}

\section{Existence of Orbits at the AI Limit}\label{sec:Existence}

In this section, we will use \Eq{AIRelation}, which is equivalent to the pair of maps \Eq{MapAsFxn} 
when $a \neq 0$, to prove the existence of AI states in subsets of the $(a,c)$ parameter space. 
The analysis will depend upon the sign of the discriminant \Eq{DeltaDefine} of the quadratic curve \Eq{EDelta}. Using  the normalization \Eq{ParamSpace}, the set $\Delta = 0$, where $\cE$ is a pair of parallel lines, corresponds to a parabola in the first quadrant of the $(a,c)$ plane that divides the plane into regions with $\Delta \gtrless 0$, as shown in \Fig{acPlane}(a). As we will see in \Sec{AIParallel} below, in the $\Delta=0$ case
it is convenient to parameterize $(a,b,c)$ with respect to the slope $m$ of the lines:
\beq{DeltaZero}
    (a,b,c) = \frac{1}{(1-m)^2} (1, -2m, m^2),  \quad (\Delta = 0),
\eeq
for $m\neq1$.
\InsertFigTwo{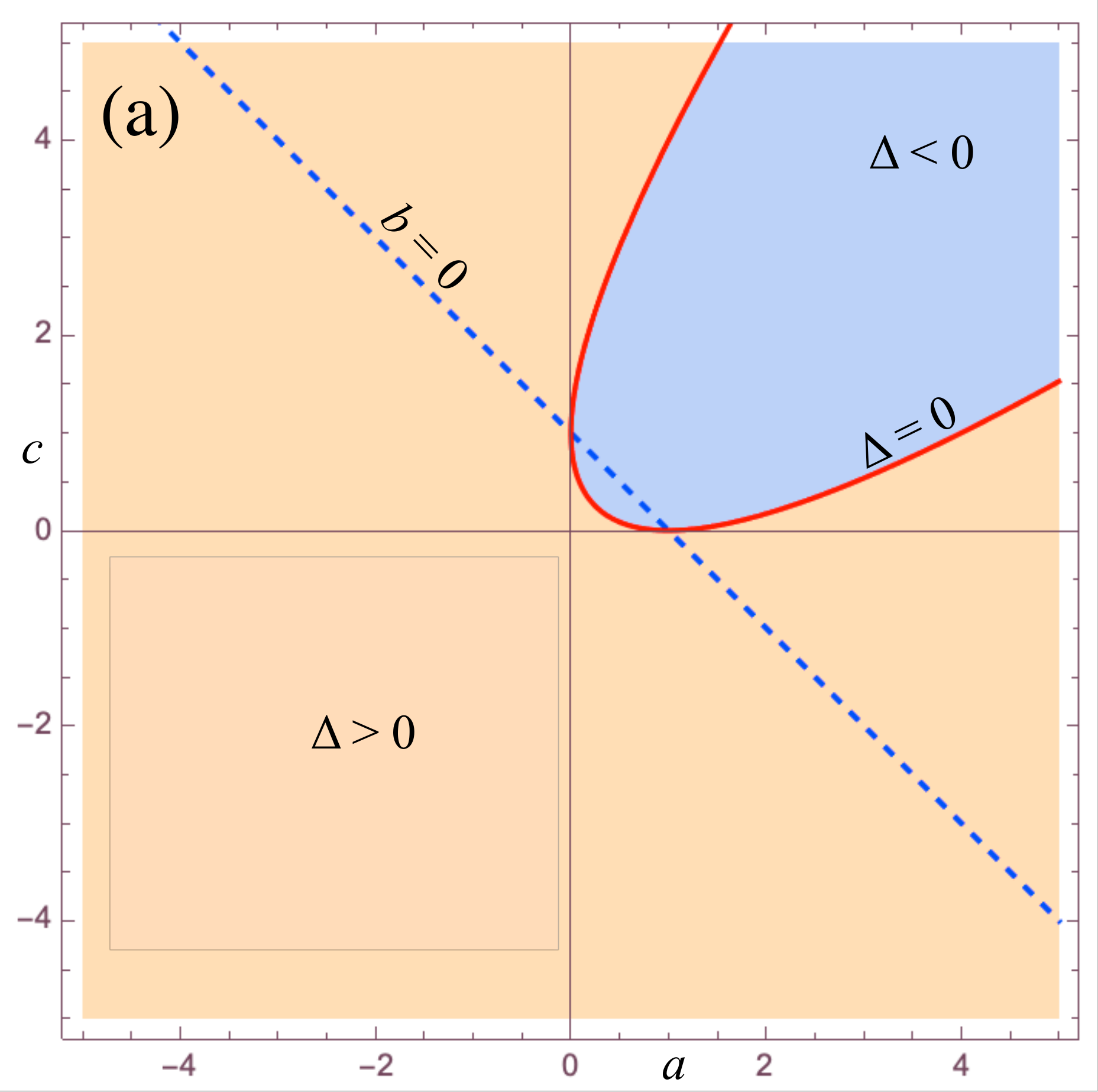}{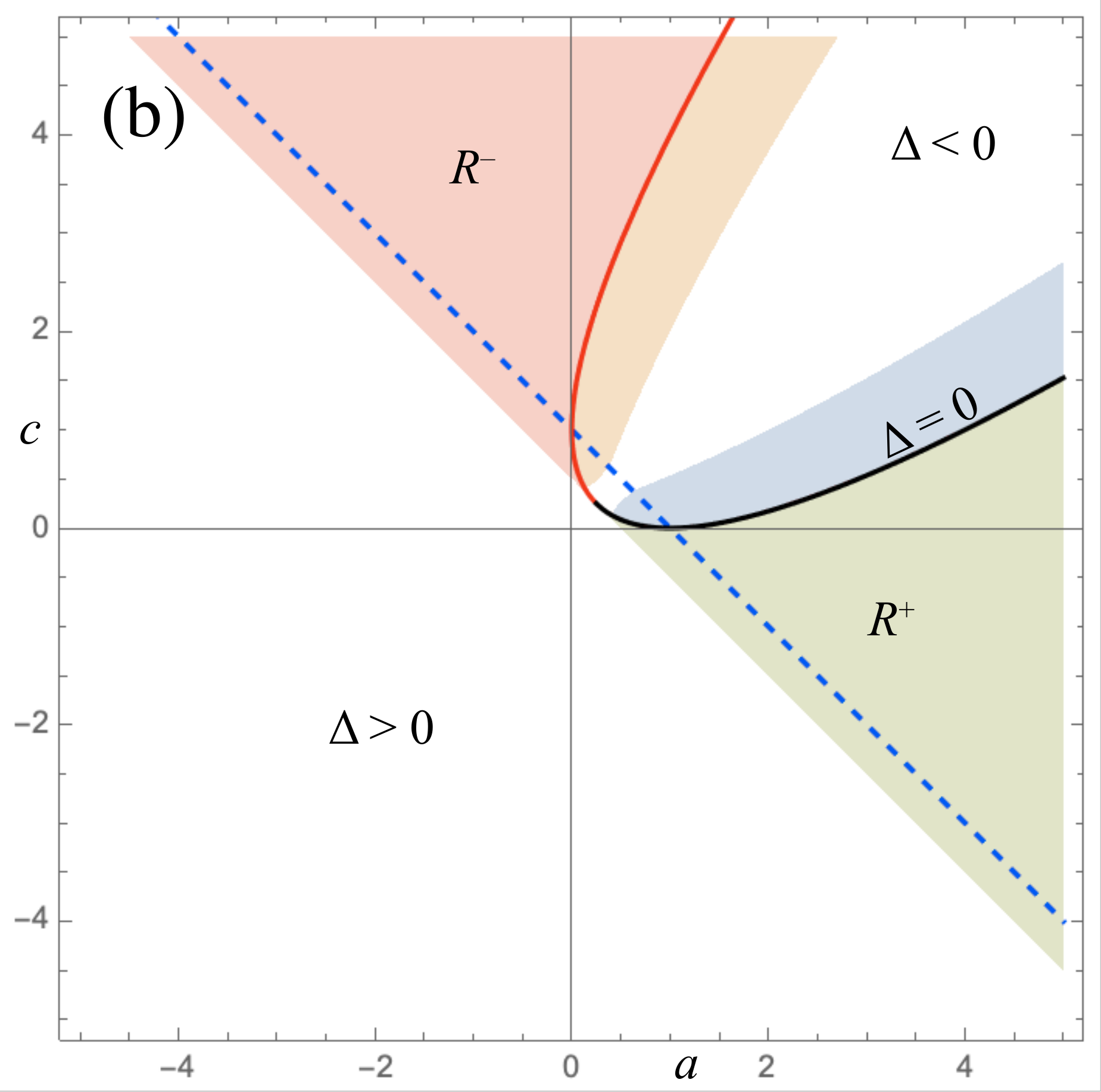}{(a) Classification of the curve \Eq{EDelta} depending upon $a$ and $c$, for $b = 1-a-c$: the blue region corresponds to ellipses, $\Delta < 0$,
the tan region to hyperbolae, $\Delta > 0$ and the red curve to parallel lines, $\Delta = 0$. 
On the blue dashed line $b=0$; $b$ is positive below the line and negative above. 
(b) Regions in the $(a,c)$-plane that result in a contraction at the AI limit in either 
forward ($\cR^+$) or backward ($\cR^-$) time. These sets are constructed in \Sec{AIParallel}-\Sec{AIHyperbola}.}
{acPlane}{.5}

The correspondence $Q(\xi_t,\xi_{t-1}) = 1$ defines an AI state $\{\xi_t: t \in \bZ\}$ provided that all pairs $(\xi_{t-1},\xi_t) \in \cE$.  First note that fixed points always exist: indeed, since $a+b+c = 1$, then 
$Q(\xi,\xi) = 1$ when $\xi = \pm 1$. Such period-one sequences thus correspond to $s = (+)^\infty$ 
and $(-)^\infty$, with $\xi_t = s_t$.
Similarly, period-two orbits correspond to points $Q(\xi_1,\xi_0) = Q(\xi_0,\xi_1) = 1$, which if $a \neq c$ implies that
\beq{PeriodTwo}
    \xi_0 = -\xi_1 = \frac{1}{\sqrt{1-2b}}
\eeq
providing $b < \tfrac12 < a+c$.
Thus in this domain, there is a unique period-two state with symbol sequence $s = (+,-)^\infty$, so that $\sign(\xi_t) = s_t$. For the special case $a = c$, every point on $\cE \setminus \{(x,x)\}$ has a corresponding period-two point.
Note however, unlike the \hen case of \Sec{Intro}, the period-two states are not simply given by their symbol sequences.

More generally, we will use contraction arguments to find a set of parameters for which there is a unique orbit for each symbol sequence. 
In the following subsections, we will separately analyze the three $\Delta$ cases.
The first step is to find parameter domains so that $f_\pm: B \to B$ for some closed interval,
\beq{BDefine}
    B = [-x^*,x^*] \subset \bR, 
\eeq
where $x^*$ depends upon $(a,c)$. In \Fig{Cobwebs}, this interval is shown as the green box, $B^2$, in $(\xi_{t-1},\xi_t)$-space. The second is to require that both maps are contractions on $B$ (i.e., have slopes with magnitude less than one). Suppose that
\beq{RCPlus}
    \cR^+ = \{(a,c): \exists\, x^*(a,c) > 0, \, f_\pm(B) \subset B, \, 
       |f'_\pm(x)|<1, \, \forall\, x \in B \}
\eeq
is a subset  on which both of these conditions are true (we will construct this region in the subsections below). Letting $\cB = B^\infty$, a simple application of the contraction mapping theorem gives:
\begin{lem}\label{lem:AIContraction} Given $(a,c) \in \cR^+$ and $\xi\in\cB$, there is a one-to-one correspondence between each sequence $s \in \Sigma$ 
and state $\{\xi\} \subset \cB$ satisfying \Eq{MapAsFxn}.
\end{lem}

\begin{proof}
Consider any sequence $\xi = \{\xi_t : t \in \bZ\} \in \cB$ and $\eta = \{\eta_t : t \in \bZ\} \in \cB$ where $\cB \subset \bR^\infty$ is the countable product of the closed interval $B$. Given some $s\in\Sigma$, define $\cF:\cB \to \cB$ by
\beq{AIMapF}
    \cF_t(\xi) = f_{s_t}(\xi_{t-1}) ,
\eeq
such that fixed points of $\cF$ are orbits of the map \Eq{MapAsFxn}. The closed interval $B$ is a compact subset of $\bR$, and thus $\cB$ is complete in the $\ell^\infty$ norm  by Tychonoff's theorem. Then
\begin{align*}
    \| \cF(\xi) -\cF(\eta)\|_\infty \le \sup_{x\in B} |f_\pm'(x)|\, \|\xi - \eta\|_\infty .
\end{align*}
When $(a,c) \in \cR^+$ and $x \in B$, then $|f'_\pm (x)| < 1$, and $\cF$ is a contraction in the $\ell^\infty$ metric with a unique fixed point. Therefore each symbol sequence has a unique corresponding orbit of \Eq{MapAsFxn}. Additionally, by construction the conditions $(a,c) \in \cR^+$ and $x \in B$ guarantees the radical of \Eq{MapAsFxn} is strictly positive, giving that every orbit of \Eq{MapAsFxn} has a unique symbol sequence $s\in\Sigma$. Therefore, there exists a one-to-one correspondence between symbol sequences and orbits of \Eq{MapAsFxn}.
\end{proof}

Examples are shown in \Fig{AILargePeriod} for an arbitrarily chosen symbol sequence of period $1000$. These AI states are obtained by simply iterating the contraction \Eq{MapAsFxn} for a 
point $\xi \in \cB$ until a convergence tolerance of $10^{-12}$ is reached.
In the following subsections, we will construct the interval $B$ and subset \Eq{RCPlus} for the three cases: $\Delta = 0$, and $\Delta \gtrless 0$.

\InsertFigThree{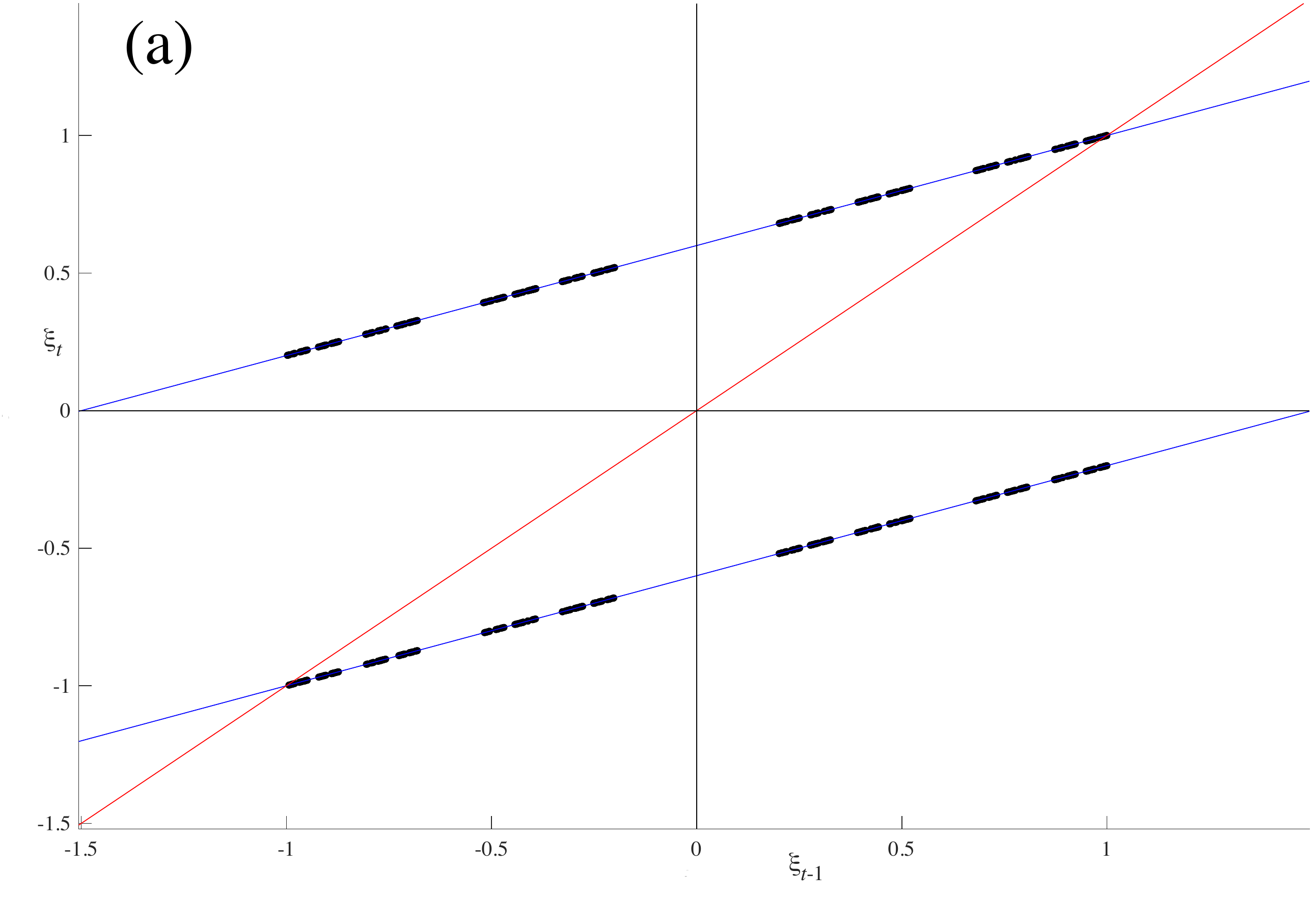}{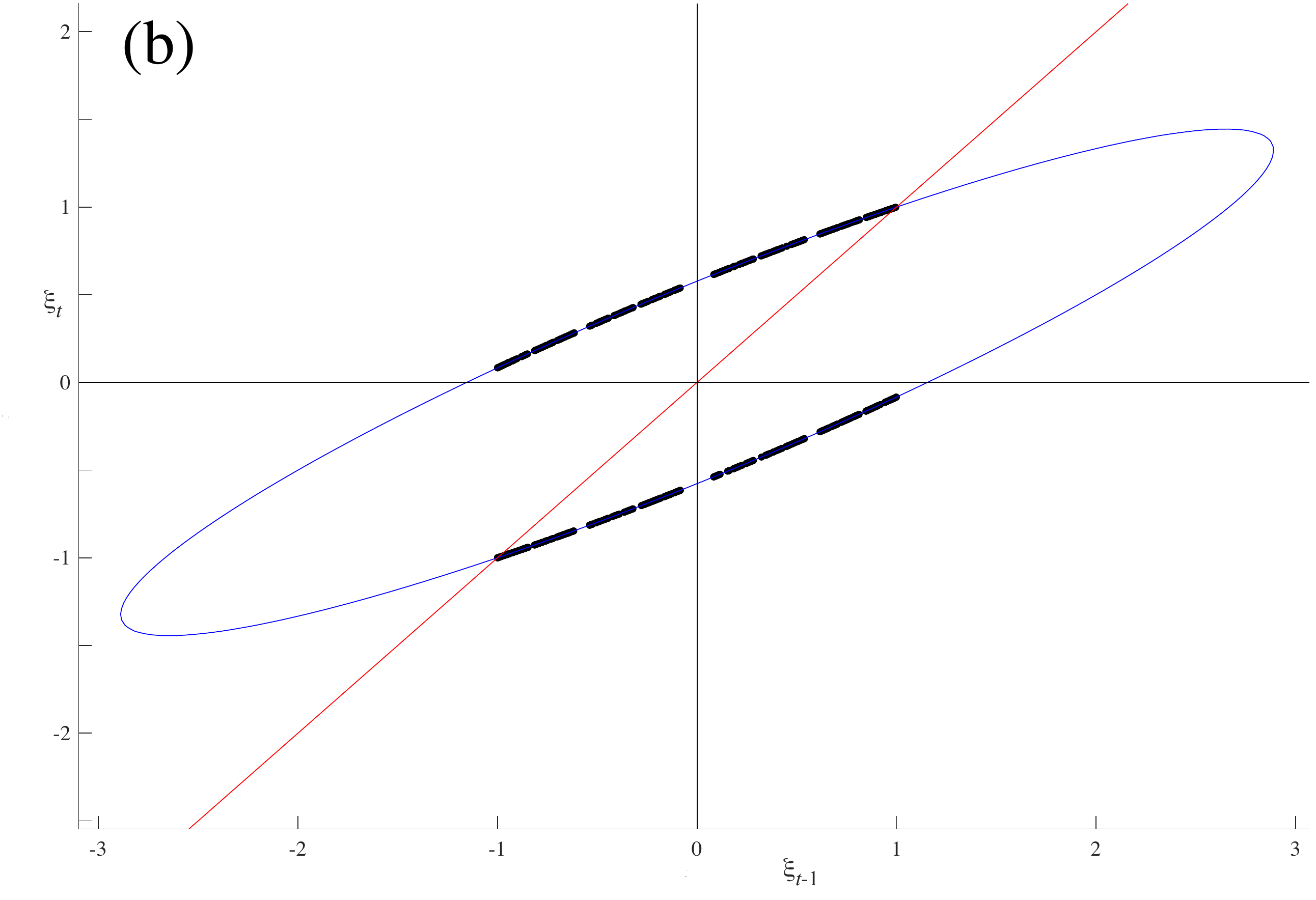}{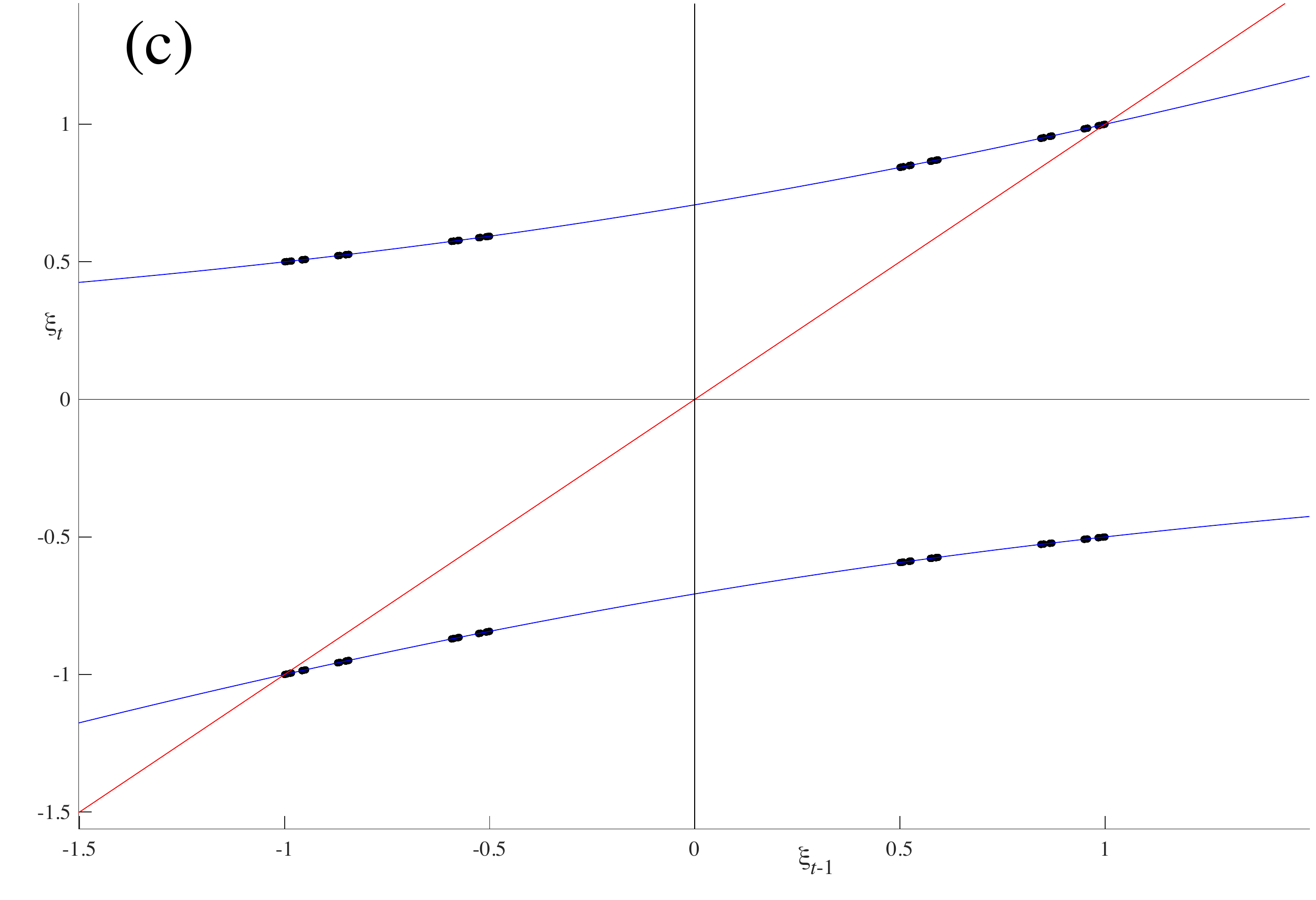}
{Orbits of period $1000$ for an arbitrarily chosen symbol sequence at the AI limit (black points) for
(a) parallel lines, (b) ellipse, and (c) hyperbola. 
Also shown are the diagonal $\xi_t=\xi_{t-1}$ (red) and the quadratic curve \Eq{EDelta} (blue).
Refer to \Fig{Cobwebs} for parameters.}{AILargePeriod}{0.3}

Note that a similar argument can be used when the maps $f_\pm$ are expanding, simply by reversing time.
More conveniently, define the inverse maps by solving $Q(\xi_t,\xi_{t-1})=1$ for $\xi_{t-1}$. By symmetry, this simply switches the coefficients $a$ and $c$ and gives the map
\[
    \xi_{t-1} = g_{s_t}(\xi_t) = \frac{1}{2c} \left(-b\xi_t+ s_t\sqrt{\Delta \xi_t^2 + 4c}\right) ,
\]
whenever $c \neq 0$.
When these ``backwards'' functions are contractions on common domains, they will also give a correspondence between symbol sequences and AI trajectories. This is equivalent to simply swapping the parameters $(a,c)$.\footnote
{Note that continuation away from the AI limit in this case would also be equivalent if we transform \Eq{QVPMap} to eliminate the parameter $\tau$ by requiring $b+2c \neq0$.}
We will call the subset of parameters $(a,c)$ for which the sequence $g_\pm$ is a contraction $\cR^-$, and the union $\cR = \cR^+ \cup \cR^-$.

\subsection{Parallel Lines} \label{sec:AIParallel}
When the discriminant vanishes, the curve $\cE$ becomes a pair of parallel lines, recall \Fig{Cobwebs}(a). This occurs on the parabola shown as the red curve in \Fig{acPlane}(a). Using the parameterization \Eq{DeltaZero}, the  map \Eq{MapAsFxn} becomes
\[
    f_{s_t}(\xi_{t-1}) = m \xi_t + s_t(1-m) ,
\]
where $m = -b/2a$ is the slope. Thus this pair of maps are contractions when $|m| < 1$, 
which corresponds to the black half of the $\Delta = 0$ parabola in \Fig{acPlane}(b)
with the boundary point $|b|=2a$ at the vertex $a=c= \tfrac14$. Similarly,
the inverse map becomes a contraction when $|m| > 1$, which corresponds to the red half of the parabola in \Fig{acPlane}(b).

We now find a minimal interval \Eq{BDefine} so that compositions of \Eq{MapAsFxn} have common domains that are mapped into themselves. This occurs for the square centered at origin whose height is the maximum magnitude of the intersection of $f_\pm$ with the diagonal $\xi_t=\xi_{t-1}$ at $1$ and the anti-diagonal $\xi_t=-\xi_{t-1}$ at the period-two orbit \Eq{PeriodTwo}. This gives an interval \Eq{BDefine} with
\beq{xStar}
    x^*=\begin{cases}
    1, & b\leq 0\\
    \frac{1}{\sqrt{1-2b}}, &  0<b<1/2
\end{cases} , \quad  \Delta = 0 .
\eeq


\subsection{Elliptic}

When $\Delta < 0$, the curve $\cE$ is an ellipse centered at origin; this corresponds to the blue region in \Fig{acPlane}(a). In order that \Eq{MapAsFxn} map an interval into itself, the height of the ellipse must be no more than its width (i.e., the range is a subset of the domain). From \Eq{MapAsFxn}, the domain of $f_\pm$ is the interval
\[
    |\xi_{t-1}| \leq 2\sqrt{\frac{a}{|\Delta|}} ,
\]
and, by symmetry, the range is 
\[
    |\xi_t| \leq 2\sqrt{\frac{c}{|\Delta|}} .
\]
Hence the range is a subset of the domain whenever $0< c< a$. Thus in this case we can restrict
\Eq{MapAsFxn} to the range, defining
\beq{xcEllipse}
    x^*= 2\sqrt{\frac{c}{|\Delta|}}, \quad \Delta < 0 .
\eeq

The map \Eq{MapAsFxn} will be a contraction whenever the magnitude of its derivative is less than one. This results in an interval in $\xi_{t-1}$ for each branch of the curve, (see \App{EllipticContractionArguments}). Constraining the intersection of these two intervals to be a subset of the range gives
\beq{EllipticContraction}
   a-2\sqrt{a}+1<c < \begin{cases}
        \frac{1}{7}(1-a+2\sqrt{16a^2-11a+2}) & \tfrac14 < a< \tfrac12 (\sqrt{5}-1)\\
        a+3-2\sqrt{a+2} & a> \tfrac12(\sqrt{5}-1)
        \end{cases},
\eeq
resulting in the blue region in $\cR^+$, seen in \Fig{acPlane}(b).

Similarly, the inverse map becomes a contraction on the tan portion of $\cR^-$, which is obtained from the formulas above by simply exchanging $a$ and $c$.

\subsection{Hyperbolic}\label{sec:AIHyperbola}
When the discriminant is positive, $\cE$ is a hyperbola. 
In this case, the smallest interval for which the range is a subset of the domain is the interval \Eq{BDefine} with $x^*$ given by \Eq{xStar}, the same as for parallel lines.

As before, \Eq{MapAsFxn} is a contraction when the magnitude of its slope is less than one. This is true everywhere if the magnitude of the slopes of the asymptotes,
\beq{m1m2Dfn}
    m_\pm=\frac{-b \pm \sqrt{\Delta}}{2a}, 
\eeq
is less than one.
To find the resulting parameter domain, first observe that
$
    m_+ m_-  
    =\frac{c}{a}. 
$
Thus if $|m_\pm|<1$ then necessarily $|c|<|a|$.
Additionally, note that the condition $|m_\pm|<1$ requires the existence of the intersection of the hyperbola with the anti-diagonal $\xi_t=-\xi_{t-1}$ at the period-two point \Eq{PeriodTwo}, implying that $b<1/2$. Elimination of $b$ and imposing $\Delta > 0$ then gives the region
\beq{HyperbolicContraction}
    \tfrac12-a<c<a-2\sqrt{a}+1,
\eeq
which is shown as the green portion of $\cR^+$ in \Fig{acPlane}(b). The argument for the inverse map is obtained by, again, exchanging $c$ and $a$ to give the red portion of $\cR^-$.

\section{Persistence of Orbits Away From the AI Limit}\label{sec:PersistOrbit}

We showed in \Sec{Existence} that there is a one-to-one correspondence between symbol sequences and orbits of the map \Eq{QVPMap} at the anti-integrable limit, provided that $(a,c) \in \cR$ and $\xi \in \cB$. In this section we will show that these orbits persist for small enough $\eps >0$ by reformulating the third-order difference equation \Eq{OurAIDiff} as a a map
\[
    T(\cdot;s):\cB_M \rightarrow \cB_M ,
\]
on sequences $\xi \in \cB_M \subset \bR^\infty$ for each symbol sequence $s \in \Sigma$ so that a fixed point, $\xi = T(\xi;s)$, corresponds to $\cL_\eps = 0$. The simplest form of such a map is obtained by formally solving \Eq{OurAIDiff} for $\xi_t$:
\beq{TMap}
    T_t(\xi;s) = \frac{1}{2a} \left( -b \xi_{t-1} +
       s_t \sqrt{ \Delta \xi_{t-1}^2 + 4a \left[1+ \eps (\xi_{t+1} + \sigma \xi_{t-1} -\delta\xi_{t-2})\right]} \right) .
\eeq
The operator is constructed so that the AI limit \Eq{MapAsFxn} is reproduced when $\eps=0$. 
Here---without loss of generality---we have assumed that $(a,c) \in \cR^+$, since
one can find a similar operator that applies in $\cR^-$ by solving the quadratic terms for $\xi_{t-1}$ which is essentially \Eq{TMap} under $a \leftrightarrow c$.

To show that this map has fixed points, we will again use the contraction mapping theorem. Thus we must show that for small enough $\eps$, there is a compact set $\cB_M \in \bR^\infty$, effectively an $M$-expansion of the cube defined using the interval \Eq{BDefine}, such that 
\[
    \|DT(\xi;s)\|_{\infty} < 1 .
\]
We will see that $DT(\xi;s)$ is an infinite matrix with a finite number of non-zero elements in each row, so its infinity norm is the supremum of the absolute row sums.

When $T(\cdot,s)$ is a contraction on $\cB_M$ then the unique fixed point of $T$ is an orbit of the map \Eq{OurAIDiff} with AI-limit defined by the symbol sequence $s$. For our analysis, it is convenient to define the parameter
\beq{gamma}
   \gamma \equiv |\eps|(1+|\sigma|+|\delta|) ,
\eeq
that measures the size of the $\eps$ dependent terms in \Eq{TMap}, and to let
\[
   \cR_M^+ \subset \{(a,c,\gamma,M) \in \bR^4 \}
\]
denote the generalization of $\cR^+$, a parameter region on which $T$ will be a contraction. This
set will be constructed for two special cases in the next two subsections. Given this region, we can easily prove:
\begin{lem}\label{lem:AIPersistence} Given $(a,c,\gamma,M) \in \cR_M^+$ and $\xi\in\cB_M$, there is a one-to-one correspondence between each symbol sequence $s \in \Sigma$ and orbit $\{\xi\} \subset \cB_M$ satisfying \Eq{TMap}.
\end{lem}

We omit a formal proof, as it is the same as \Lem{AIContraction}. Again, this argument also applies in a region
$\cR_M^-$ obtained by exchanging $a \leftrightarrow c$.

As a result of \Lem{AIPersistence}, orbits will not undergo any bifurcations when parameters lie in $\cR_M = \cR_M^+ \cup \cR_M^-$. This ``no-bifurcations'' concept is a simple consequence of uniqueness \cite{Sterling99}.  

In the next two subsections, we consider two special cases and construct the region $\cR_M$ so that $T$ is indeed a contraction. While these cases are clearly not exhaustive, they provide insight into the utility of the AI approach and demonstrate the nontrivial nature of the problem. In \Sec{Bzero} we set $b = 0$, and treat the elliptic, hyperbolic and parallel lines cases together. In \Sec{Parallel} we set $\Delta = 0$, so that the set $\cE$, \Eq{EDelta}, corresponds to a pair of parallel lines. 
In both cases, we will demonstrate that simple iteration of the contraction $T(\xi;s)$ with parameters chosen appropriately, will give orbits of the map \Eq{QVPMap} to any desired precision.

\subsection{Vanishing \textit{b}}\label{sec:Bzero}

As shown in \Fig{acPlane}(b), a portion of the line $b=0$, where $a+c = 1$, is in $\cR^+$. This line intersects all three $\Delta$ cases: parallel lines, elliptic, and hyperbolic. Setting $b=0$ in \Eq{TMap} gives
\beq{VanBOp}
   \xi_t=T^{0}_t(\xi;s) = s_t \sqrt{\frac{1}{a}\left[1 - c\xi_{t-1}^2 + \eps (\xi_{t+1}+\sigma \xi_{t-1}-\delta\xi_{t-2})\right]} ,
   \quad (b=0).
\eeq
Note that 
\[ 
    T^{0}_t(\xi;s)|_{\eps = 0} = f_{s_{t}}(\xi_{t-1}) = s_{t}\sqrt{\frac{1}{a}(1-c\xi_{t-1}^2)} ,
\]
reproduces the map \Eq{MapAsFxn} for this case.

Since $a=1-c>0$, the fixed points at this AI limit are simply $\xi_t = s_t$. The expansion of $\cB$ by $M \in [0,1)$ is then defined to be the cube
\beq{B^b0_M}
      \cB^0_M(s) \equiv \{ \xi \in \bR^\infty \mid \|\xi-s\|_{\infty}\leq M \}, \quad (b=0),
\eeq
for each $s \in \Sigma$.
Given $\xi \in \cB^0_M$ and some algebra (see \App{VanBPersistenceArgument}),  $T^{0}(\xi;s) \in \cB^0_M$ when
\beq{bVanishes1}
    \gamma \leq  \frac{1}{1+M} \min \left((a-1,\,\, 1- a(1-M)^2 \right) -|1-a|(1+M) .
\eeq
Additionally, requiring $\|DT^{0}(\xi;s)\|_{\infty}<1$ leads to the condition
\beq{bVanishes2}
     \gamma < 2a(1-M) - 2|1-a| (1+M) .
\eeq
These conditions imply that $T^{0}: \cB^{0}_M \to \cB^{0}_M$ is a contraction when parameters lie within a region in $(\gamma,M,a)$-space, pictured in \Fig{VanishingBRegions}. Thus, given appropriately chosen parameters, every AI orbit continues uniquely to a fixed point of $T^{0}$, which is an orbit of the map \Eq{QVPMap}.

\InsertFigTwo{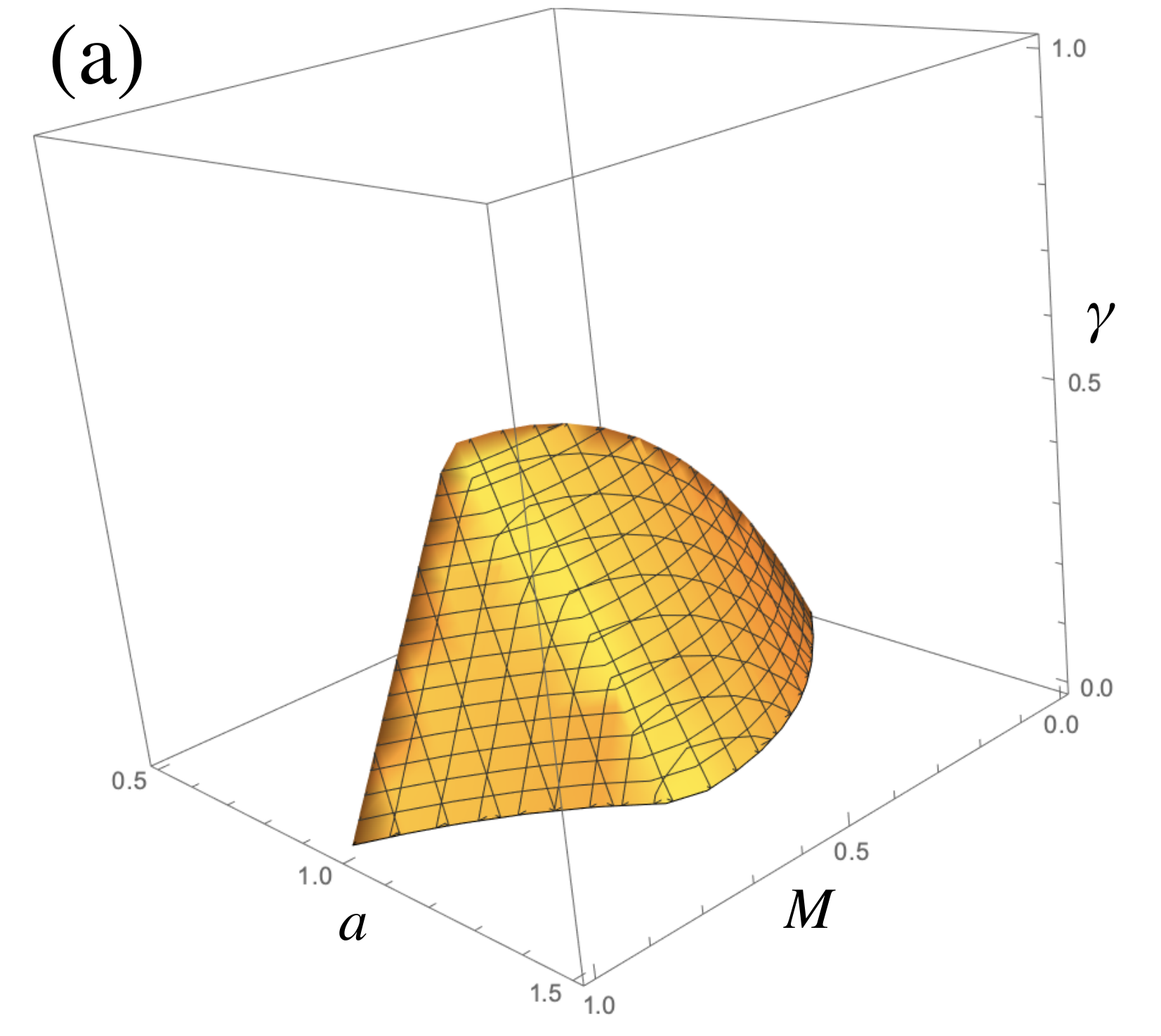}{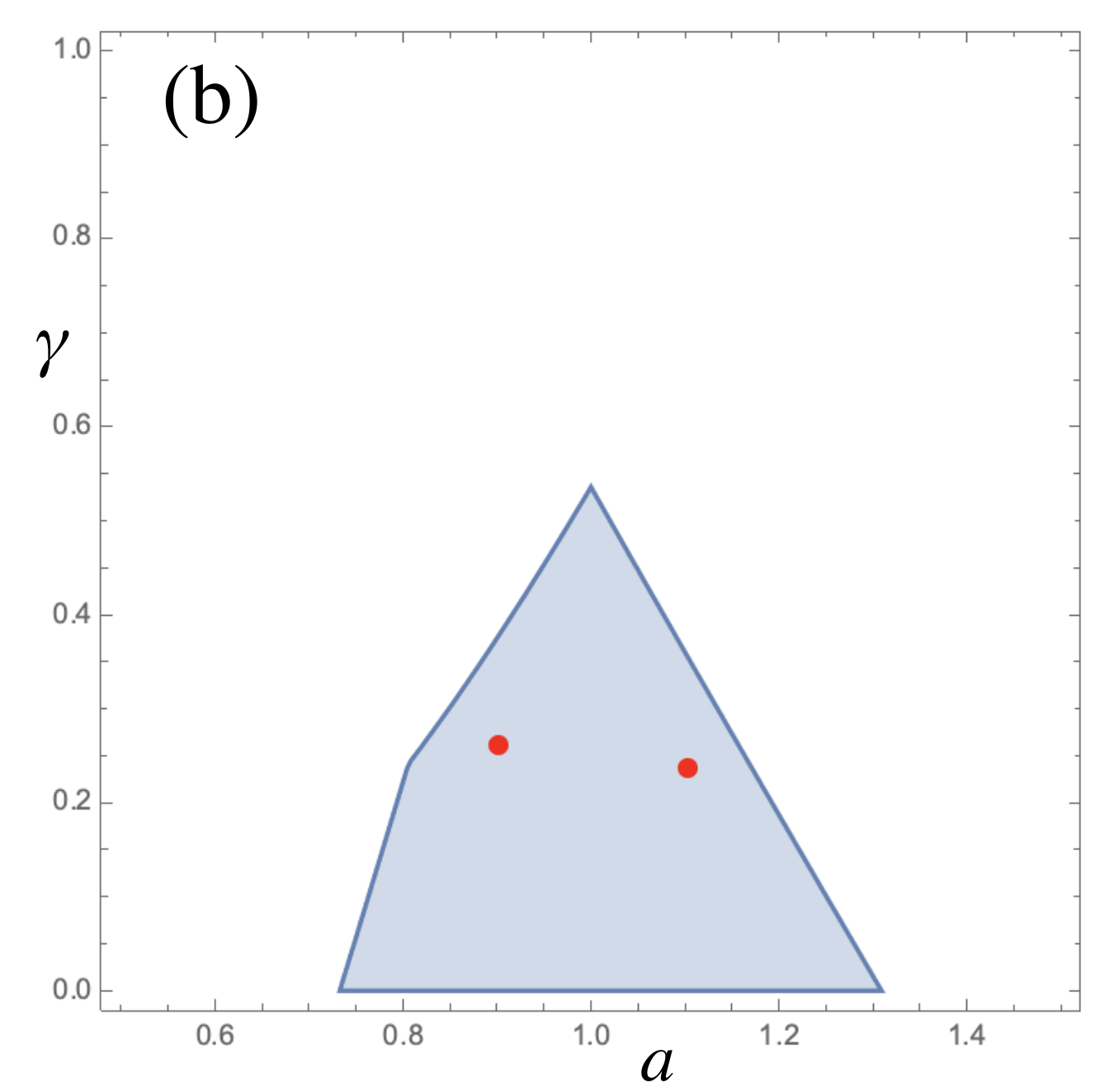}{(a) The region $\cR_M^+$ in $(a,M,\gamma)$ where $T^{0}$ is guaranteed to be a contraction on $\cB^0_M$, and (b) the projection of this region to the $(a,\gamma)$-plane. The red points show the parameters used in \Fig{Tb0Examples}(a) and (b) when $\eps=0.15$. }{VanishingBRegions}{0.5}

Two examples are shown in \Fig{Tb0Examples}. The blue curves are the AI limit curve, \Eq{EDelta}, which, in panel (a), is an ellipse, and in panel (b) is a hyperbola. In both cases, an arbitrarily chosen symbol sequence of period-$250$ and sequence $\xi \in \cB^0_M$ is iterated under $T^0$ until a convergence tolerance of $10^{-12}$ is reached to obtain the orbits.
Orbits with three $\eps$-values are shown; 
at the smallest value, $\eps=0.01$, the orbit essentially lies on the four points $(\xi_{t-1},\xi_t) = (s_{t-1},s_t)$, and as $\eps$ grows the orbits move away from these values. The largest value, $\eps = 0.15$, in the figure corresponds to (a) $\gamma = 0.2625$ and (b) $\gamma = 0.2375$. The $(a,\gamma)$-values for both of these examples are shown as the red points in \Fig{VanishingBRegions}(b), thus in either case the map $T^{0}$ is indeed a contraction.

\InsertFigTwo{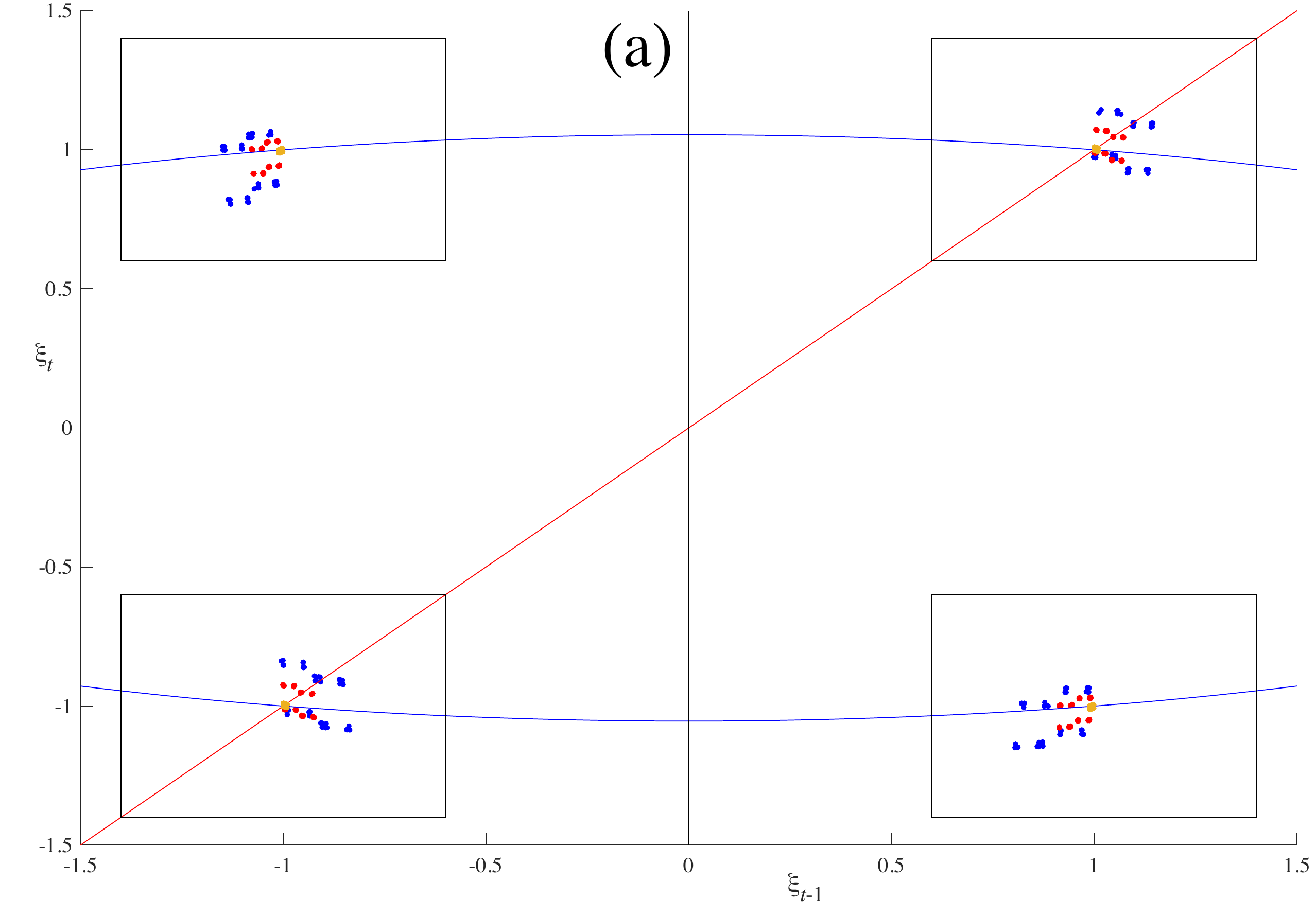}{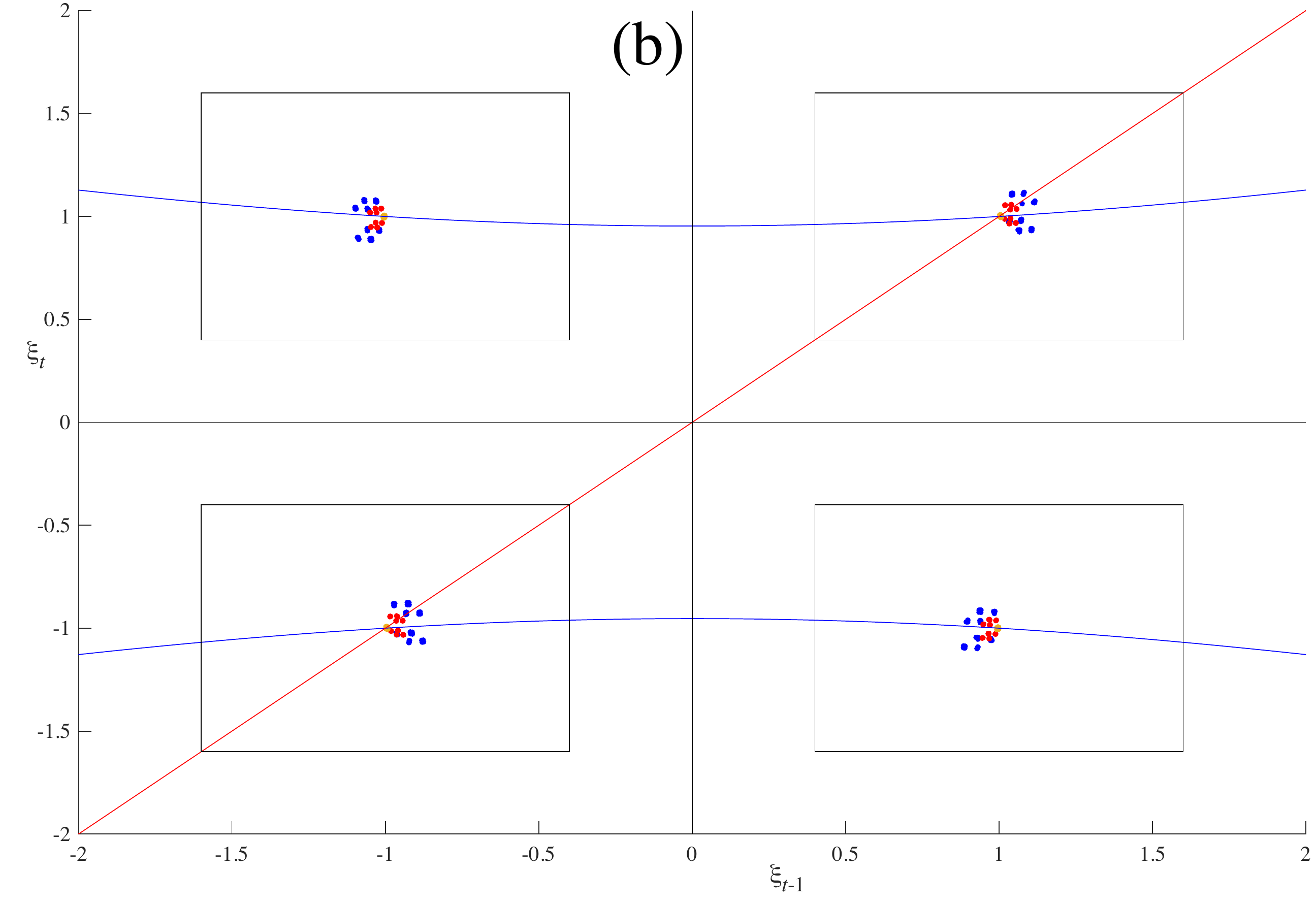}{An arbitrarily chosen sequence $s$ of period-250 is iterated under the contraction $T^0$ to obtain orbits of \Eq{QVPMap} for two cases with $b = 0$:
(a) $(a,c,\sigma,\delta) = (0.9,0.1,0.5,0.25)$ and 
(b) $(a,c,\sigma,\delta) = (1.1,-0.1,0.25,\tfrac13)$.
Three $\eps$-values are shown, $\eps = 0.01$ (yellow), $0.075$ (red), and $0.15$ (blue).
An initial sequence $\xi \in \cB^0_M$, was chosen at random, setting $M = 0.4$ in (a) and $0.6$ in (b). These sequences are iterated until they converge within an error of $10^{-12}$. The orbits are shown in the $(\xi_{t-1},\xi_{t})$-plane with the quadratic curve \Eq{EDelta} (blue), the space $B_M$ (four black boxes), and the diagonal $\xi_t=\xi_{t-1}$ (red). 
The red points in \Fig{VanishingBRegions}(b) show the $(a,\gamma)$-values when $\eps=0.15$ for these examples.}{Tb0Examples}{0.5}

\subsection{Parallel Lines}\label{sec:Parallel}
When $\Delta = 0$, the quadratic curve \Eq{EDelta} is a pair of parallel lines and the coefficients $(a,b,c)$ can be parameterized with the slope $m=-\frac{b}{2a}$, recall \Eq{DeltaZero}.
Rewriting the difference equation \Eq{OurAIDiff} with this parameterization gives
\[
     \eps(\xi_{t+1} + \sigma \xi_{t-1} - \delta \xi_{t-2})=\frac{1}{(m-1)^2}(\xi_t-m\xi_{t-1})^2-1 .
\]
The operator $T$ from \Eq{TMap} then has the form
\beq{ParallelLinesOperator}
    \xi_t = T_t^{\parallel}(\xi;s) = m \xi_{t-1} + s_t(1-m)\sqrt{ 1+ \eps(\xi_{t+1} + \sigma \xi_{t-1} - \delta \xi_{t-2})}, \quad (\Delta = 0),
\eeq
for any $s \in \Sigma$.
This reduces to \Eq{MapAsFxn} for $\Delta = 0$ when $\eps = 0$. Using the bound $x^*$ in \Eq{xStar},
which is appropriate for $\Delta = 0$, we define the cube
\beq{B_M^||}
    \cB_M^{\parallel}=\{\xi \in \bR^\infty \mid \|\xi\|_\infty \leq x^*+M\}, \quad (\Delta = 0) .
\eeq

After some algebra (see \App{GenParallelLinesPersistence}), we see that $T^{\parallel}$ maps 
$\cB_M^{\parallel}$ into itself when 
\beq{ParallelLines1}
    \gamma \leq \frac{(1-|m|)^2(x^*+M)^2-(1-m)^2}{(1-m)^2(x^*+M)}.
\eeq
The bound $\|DT^\parallel(\xi;s)\|_{\infty}<1$ requires
\beq{ParallelLines2}
    \gamma^2(1-m)^2-4(1-|m|)^2(1-\gamma(x^*+M)) < 0.
\eeq
Combining these inequalities gives a region in $(\gamma, M,m)$-space shown in \Fig{ParallelLinesRegion}. Any set of parameters chosen within this region results in the operator $T^{\parallel}$ being a contraction on $\cB^\parallel_M$, implying the existence of a unique orbit of \Eq{QVPMap} for any sequence $s \in \Sigma$.

\InsertFigTwo{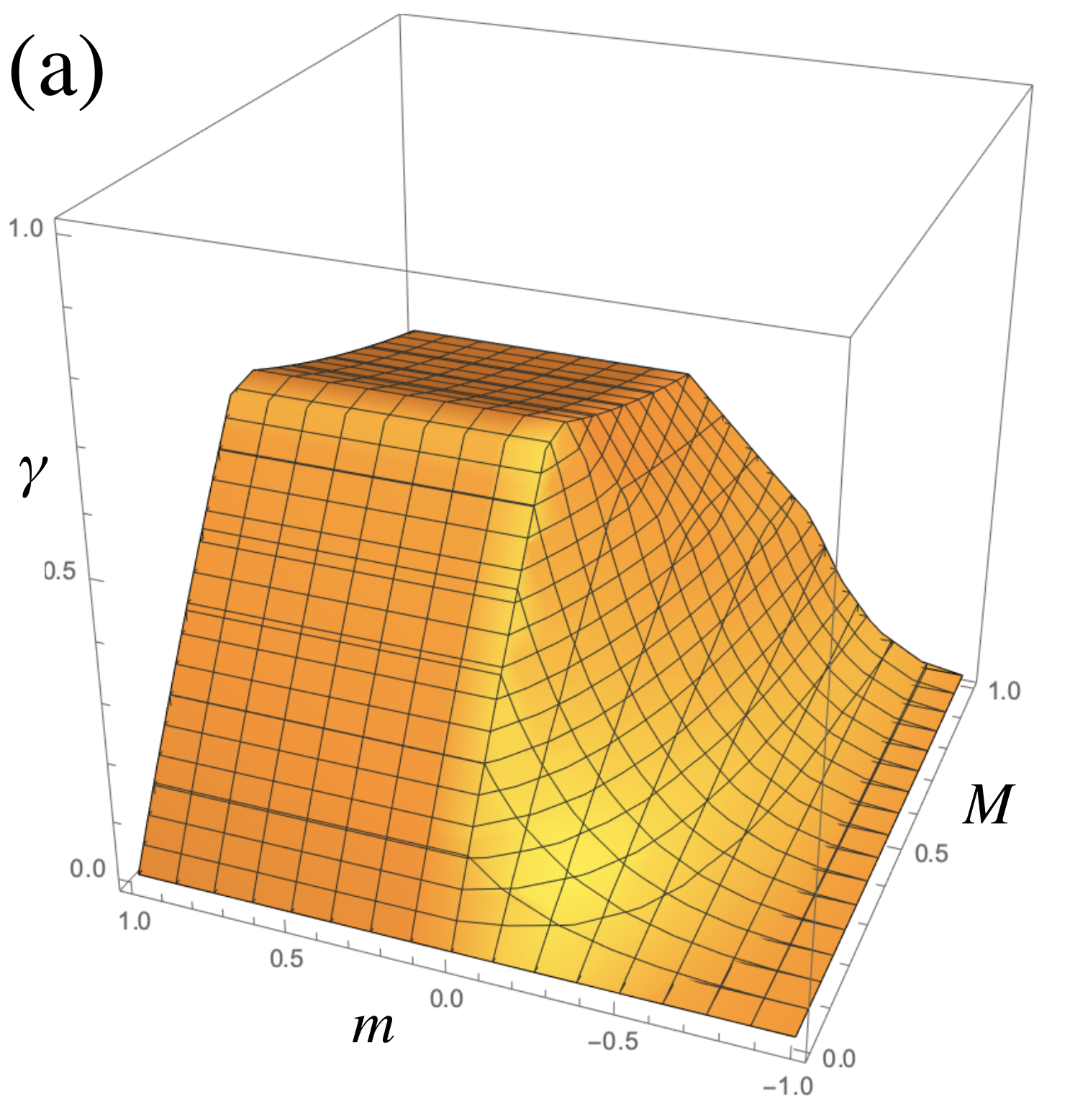}{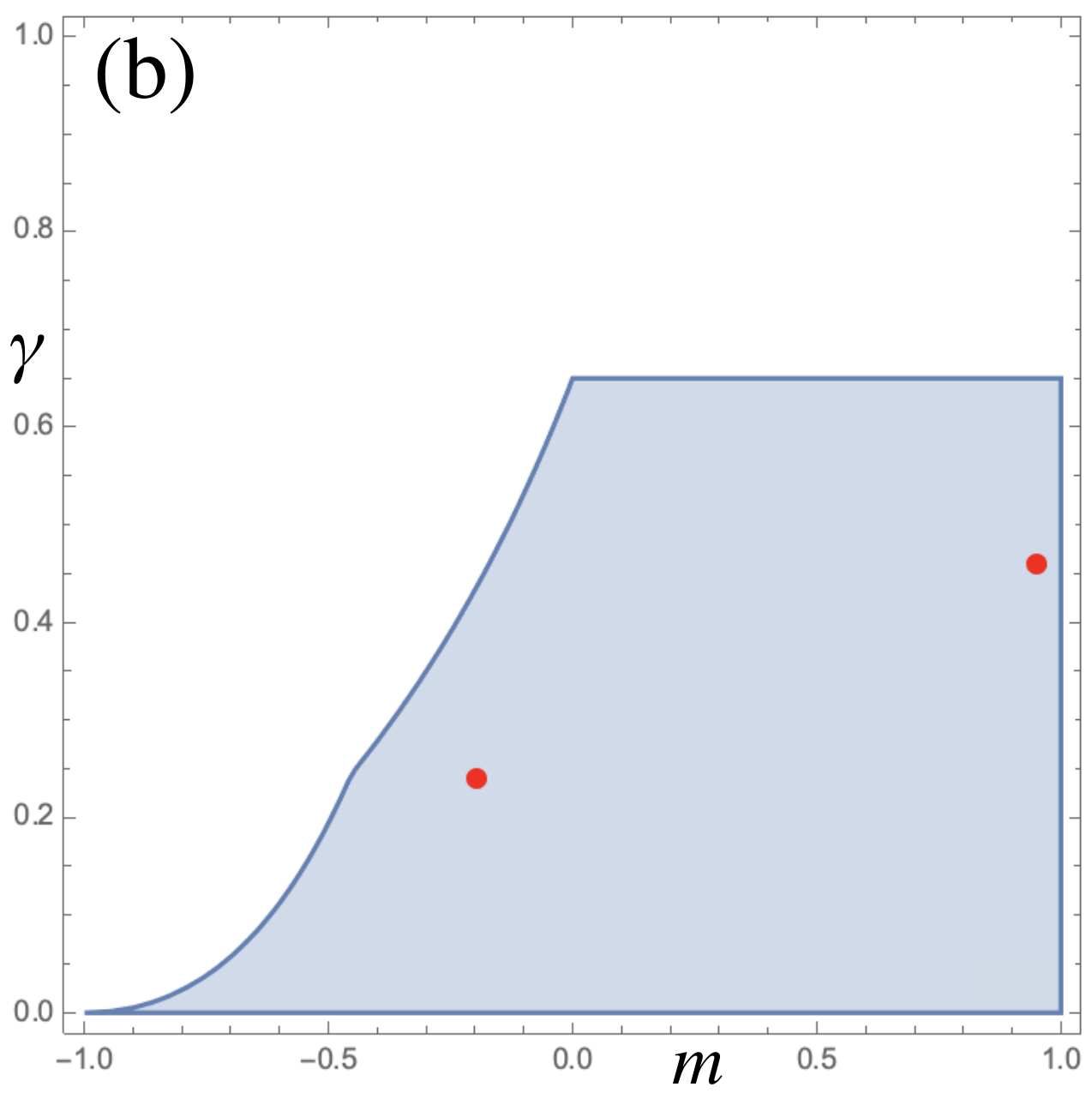}{(a) The region in $(m,M,\gamma)$-space, where $T^{\parallel}:\cB^\parallel_M \rightarrow \cB^\parallel_M$ and $\|DT^{\parallel}\|_{\infty} < 1$, and (b) the projection of the region onto the $(m,\gamma)$-plane. The red points show the $(m,\gamma)$-values used in \Fig{T__Examples}(a) and (b) when $\eps=0.2$. }{ParallelLinesRegion}{0.5}

Similar to above, we show two examples in \Fig{T__Examples}.
In each case we arbitrarily choose an initial sequence of period-250 and $\xi\in\cB_M^\parallel$ and iterate it under the map $T^\parallel$ until a convergence tolerance of $10^{-12}$ is obtained. Orbits with three $\eps$-values are shown. For the smallest value, $\eps = 0.01$, the orbits essentially fall on the set \Eq{EDelta}, the blue parallel lines, but the orbit is nevertheless nontrivial: it is not simply the symbol sequence.
For the case shown in \Fig{T__Examples}(b), $m$ is larger and the orbit lies in a smaller interval but stretches out more along
the lines.
When $\eps = 0.2$, $\gamma = 0.24$ in (a) and $\gamma = 0.46$ in (b). 
The corresponding $(m,\gamma)$-values, shown as red points in \Fig{ParallelLinesRegion}(b), 
lie within the region where $T^{\parallel}$ is guaranteed to be a contraction.

\InsertFigTwo{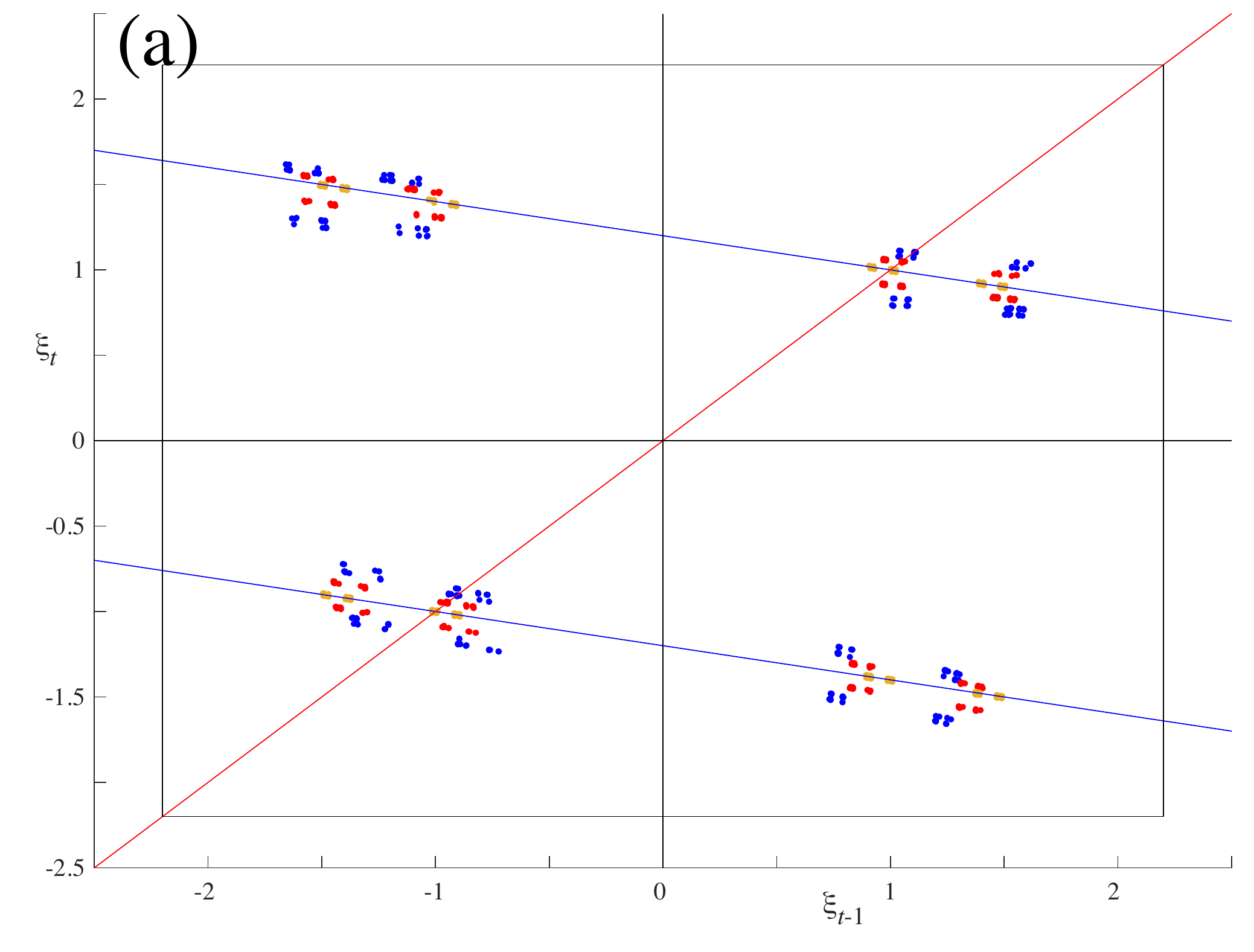}{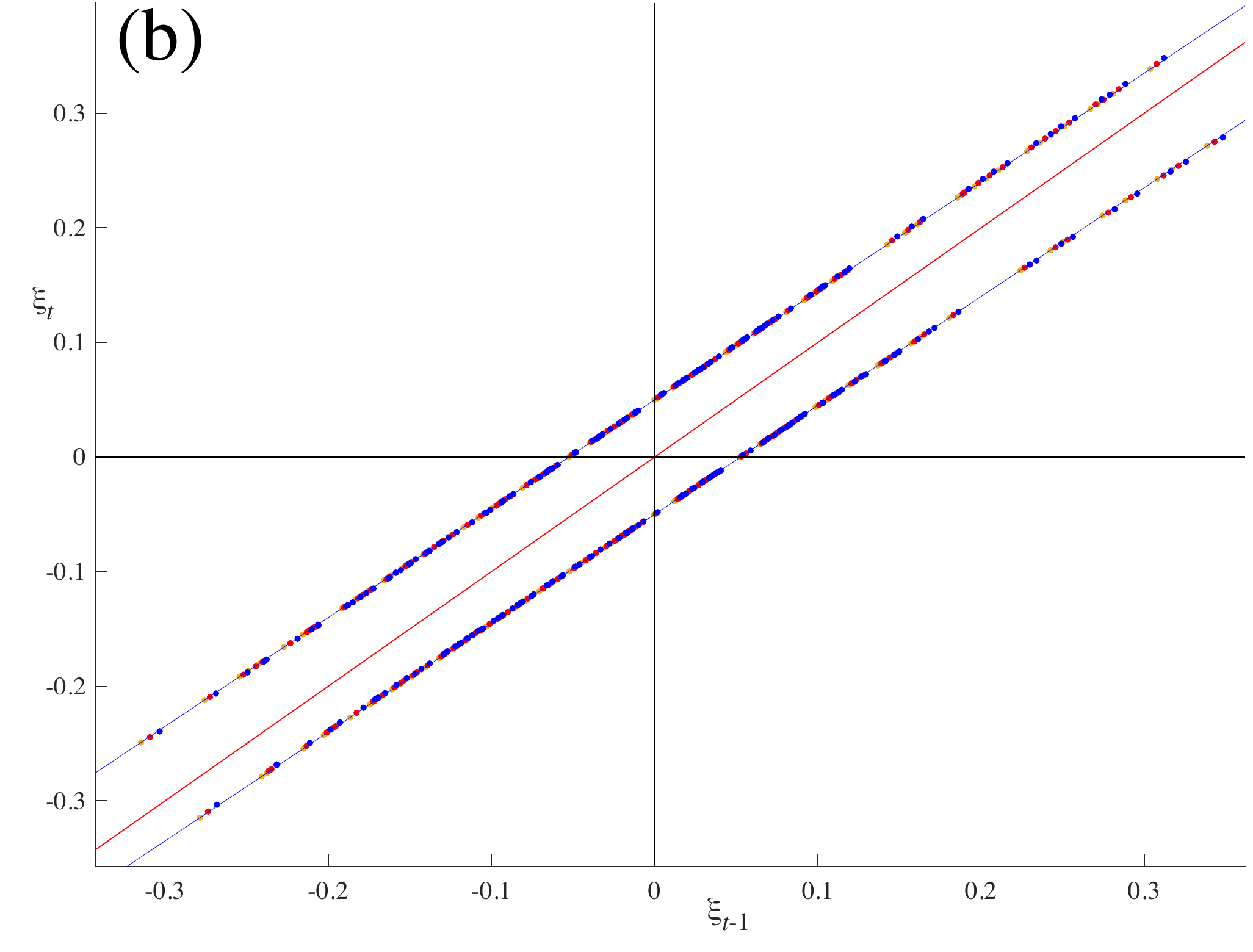}{Orbits of \Eq{QVPMap} of period-250 with parameters (a) $(m,\sigma,\delta)=(-0.2,0.1,0.1)$ and (b) $(m,\sigma,\delta)=(0.95,0.7,0.6)$,
for $\eps=0.01$ (yellow), $0.1$ (red), and $0.2$ (blue).
These are obtained from a randomly chosen sequence $s$ and an initial sequence $\xi \in \cB^\parallel_M$, setting $M = 0.7$ in (a) and $0.5$ in (b). The sequences are iterated using $T^{\parallel}$ until they converge within an error of $10^{-12}$. The orbits are shown in the $(\xi_{t-1},\xi_{t})$-plane with the quadratic curve \Eq{EDelta} (blue), the space $B_M$ (black box; not seen in (b) as the orbit lies in a much smaller interval), and the diagonal $\xi_t=\xi_{t-1}$ (red).  The red points in \Fig{ParallelLinesRegion}(b) show the $(m,\gamma)$-values when $\eps=0.2$ for these examples. }{T__Examples}{0.5}

\section{Continuation from the AI Limit}\label{sec:Continuation}

In this section, we implement a continuation algorithm to find periodic orbits of \Eq{QVPMap} as $\eps$ grows from zero by reformulating the rescaled difference equation \Eq{OurAIDiff} to use a predictor-corrector method. An orbit of period-$n$, i.e., a sequence
\[
    \xi \in  \{\xi \in \bR^\infty \mid \xi_{t+n} \equiv \xi_{t}, \, \forall t \in \bZ \} \simeq \bR^n ,
\]
must be a zero of the function $\cG: \bR^n \times \bR \to \bR^n$ defined by
\[
   \cG(\xi,\eps) =\left( \cL_\eps(\xi_1,\xi_0,\xi_{n-1},\xi_{n-2}), \cL_\eps(\xi_2,\xi_1,\xi_{0},\xi_{n-1}), \ldots ,
    \cL_\eps(\xi_{0},\xi_{n-1},\xi_{n-2},\xi_{n-3}) \right) ,
\]
where $\cL_\eps$ is given in \Eq{OurAIDiff}.
Given a point $(\xi^k,\eps^k)$ such that $G(\xi^k,\eps^k) = 0$ and $\ell$ a predetermined arclength step size, we use a standard pseudo-arclength continuation algorithm
\cite[Sec. 1.2.3]{Krauskopf07}, to solve the system
\bsplit{Arclength}
    &\cG(\xi^{k+1}, \eps^{k+1}) = 0 ,\\
    &\dot{\xi}^k(\xi^{k+1}-\xi^k)+\dot{\eps}^k(\eps^{k+1}-\eps^k)=\ell ,
\esplit
for orbit $\xi^{k+1}$ at parameter $\eps^{k+1}$. 

At each step, a hyperplane orthogonal to the direction vector $(\dot{\xi}^k,\dot{\eps}^k)$ is set at a distance $\ell$ from $(\xi^k,\eps^k)$ to `predict' the location of $(\xi^{k+1},\eps^{k+1})$. Then, to `correct' we use a quasi-Newton method \cite{Allgower90} that uses Broyden's 
iterative formula to approximate the Jacobian of \Eq{Arclength} and a $QR$-decomposition
to find its inverse. 

The algorithm is initialized with the orbit at the AI limit, $(\xi^{0},0)$, and direction vector $(\dot{\xi}^0,0.005)$ so
that $\dot{\xi}^0$ solves the first $n$ rows of 
\beq{tangentVector}
    \begin{pmatrix}
        \partial_\xi \cG(\xi^k,\eps^k) & \partial_\eps \cG(\xi^k,\eps^k) \\
        \dot{\xi^k}^T & \dot{\eps}^k
    \end{pmatrix}
    \begin{pmatrix}
        \dot{\xi}^{k+1}\\
        \dot{\eps}^{k+1}
    \end{pmatrix}
    = \begin{pmatrix}
         0\\ 1
    \end{pmatrix} .
\eeq
Subsequently, each new direction vector $(\dot{\xi}^{k+1},\dot{\eps}^{k+1})$ is found by solving the full system \Eq{tangentVector} to obtain a normalized tangent vector.
 
Other inputs include the parameters $a,b,c,\sigma,\delta$ and an initial arclength step size $\ell$, which varies depending on the period $n$ and values of the parameters. We allow for adjustments in $\ell$: if a solution jumps `too far away' (e.g., converges to another solution), the step is decreased by a factor of two. After each successful step, $\ell$ is reset to the original predetermined value. 
The continuation runs until $\eps$ leaves a predetermined interval, typically $\eps \in [0,2]$, and the Broyden iteration stops when either $\|\cG\|_\infty < 10^{-12}$
or after a maximum number of steps is reached; here set to $150$.

Below we consider four examples with the parameters given in \Tbl{4ExTab}; one has parallel lines, $\Delta = 0$, and three have $b=0$. These cases are labelled by the symbols listed in the first column of the table, and will be referred to by these symbols throughout the rest of our discussion. 

In the case labeled (H\'e) the map \Eq{QVPMap} has Jacobian $\delta = 0$ and is effectively a two-dimensional map (see \App{EmbeddedHenon}). For the parameters in \Tbl{4ExTab}, the resulting dynamics of this case is equivalent to the standard H\'enon map \Eq{HenonMap}, with the 2D Jacobian $\delta_H = \sigma = -0.3$, so we call it an ``embedded H\'enon'' map. When $\eps \approx 0.845$ this map has the strange attractor originally studied in \cite{Henon76}.

\begin{table}[ht]
\centering
\begin{tabular}{c |c c c c c c c}

    Case  & $a$ & $b$ & $c$ & $\Delta$ & $\sigma$ & $\delta$ & $\eps_N$ \\
    \hline
    ($\parallel$) & 25/9 & -20/9 & 4/9 & 0 & 0.1 & 0.1 & 0.5416\\
  
    (E) & 0.9 & 0 & 0.1 & -0.36 & 0.5 & 0.25 & 0.2143\\

    (H\'{e}) & 1 & 0 & 0 & 0 & -0.3 & 0 & 0.4122\\

    (VP) & 1.25 & 0 & -0.25 & 1.25 & 0.1 & 1 & 0.0481\\

\end{tabular}
\caption{Parameters for the four cases followed in \Sec{Continuation}, henceforth referred to by their symbols, listed in the first column. The first has an AI limit with parallel lines, $\Delta = 0$ ($\parallel$). The remaining three have $b=0$: an elliptic AI limit (E), an embedded H\'enon map (H\'{e}), and a hyperbolic AI limit that is volume-preserving (VP). The last column, $\eps_N$, is an upper bound for which \Lem{AIPersistence} guarantees `no bifurcations', i.e., all AI states uniquely continue to orbits for $\eps<\eps_N$.}
\label{tbl:4ExTab}
\end{table}

Results of the continuation for the four cases of \Tbl{4ExTab} are shown in 
Figs.~\ref{fig:BifurcationDiagrams}-\ref{fig:3DLargePeriodExamples}.

The first figure shows bifurcation diagrams for all orbits up to period five 
(a total of $14$ orbits). 
Each orbit is labeled by its AI symbol sequence. 
Note that for each case there is a bifurcation-free region that contains $0 < \eps < \eps_N$,
where $\eps_N$ (the last column in \Tbl{4ExTab}) is the upper bound, obtained in \Sec{PersistOrbit},
so that \Eq{TMap} is guaranteed to be a contraction.
For the first three cases, all of the orbits shown in \Fig{BifurcationDiagrams}---except the fixed points---undergo saddle-node or period-doubling bifurcations within the range $\eps_N < \eps < 2$, with the exception of the (VP) case, where both period-three orbits, $\{--+\}$ and $\{-++\}$, continue beyond $\eps = 2$, as can be seen in \Fig{BifurcationDiagrams}(d). These two orbits eventually undergo a saddle-node bifurcation at $\eps \approx 10.89$. Note that the two fixed points are given by
\[
    \{ \pm \}^\infty:\xi = \tfrac12 \left((1+\sigma-\delta)\eps \pm \sqrt{4+(1+\sigma-\delta)^2\eps^2}\right).
\]
Thus they will collide in a saddle-node when the radicand is zero, or equivalently when $\alpha = -1/\eps^2 = \tfrac14 (1+\sigma-\delta)^2$. Therefore this bifurcation will not occur for the assumed range $\alpha <0$.

\InsertFigFour{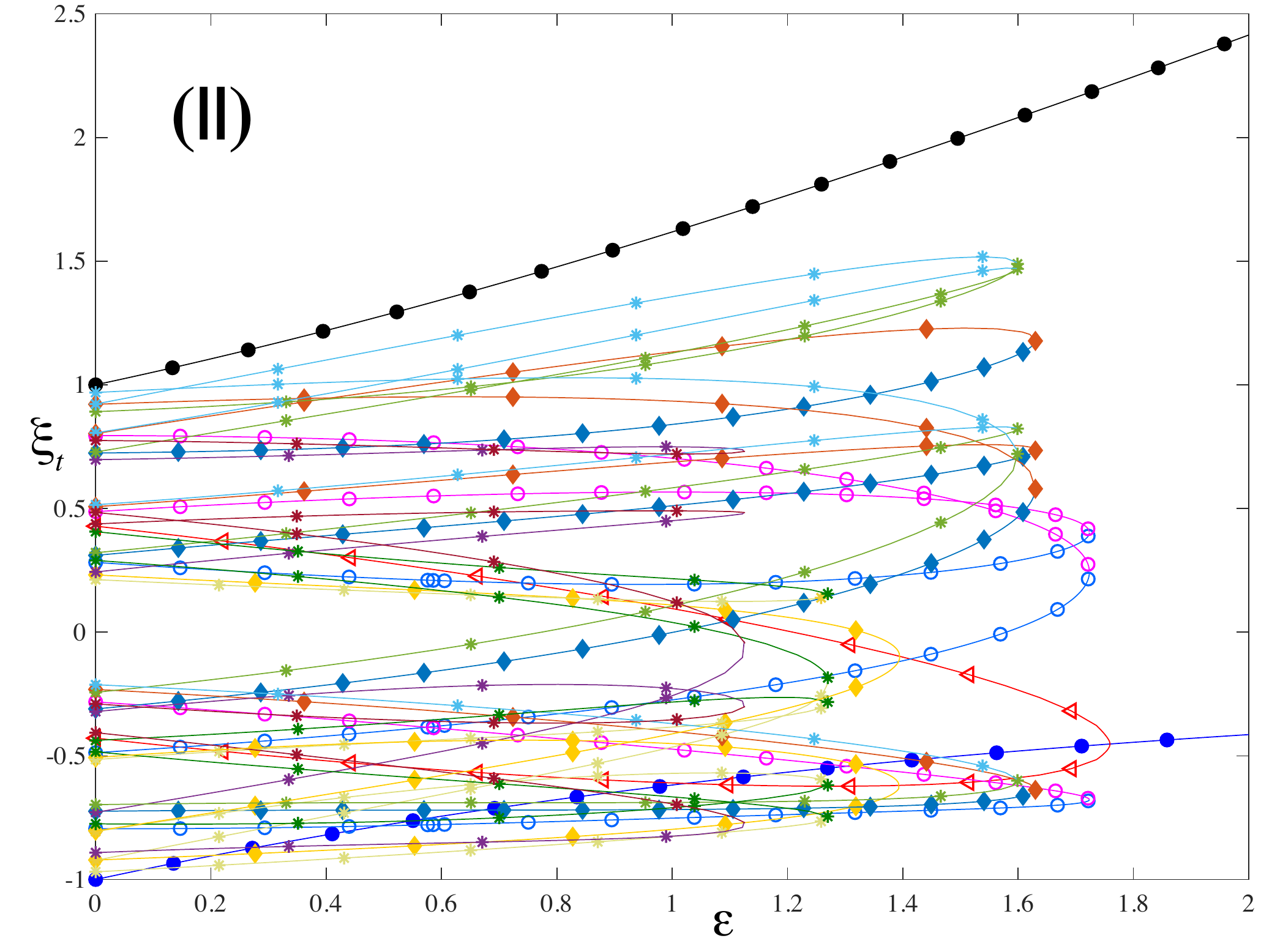}{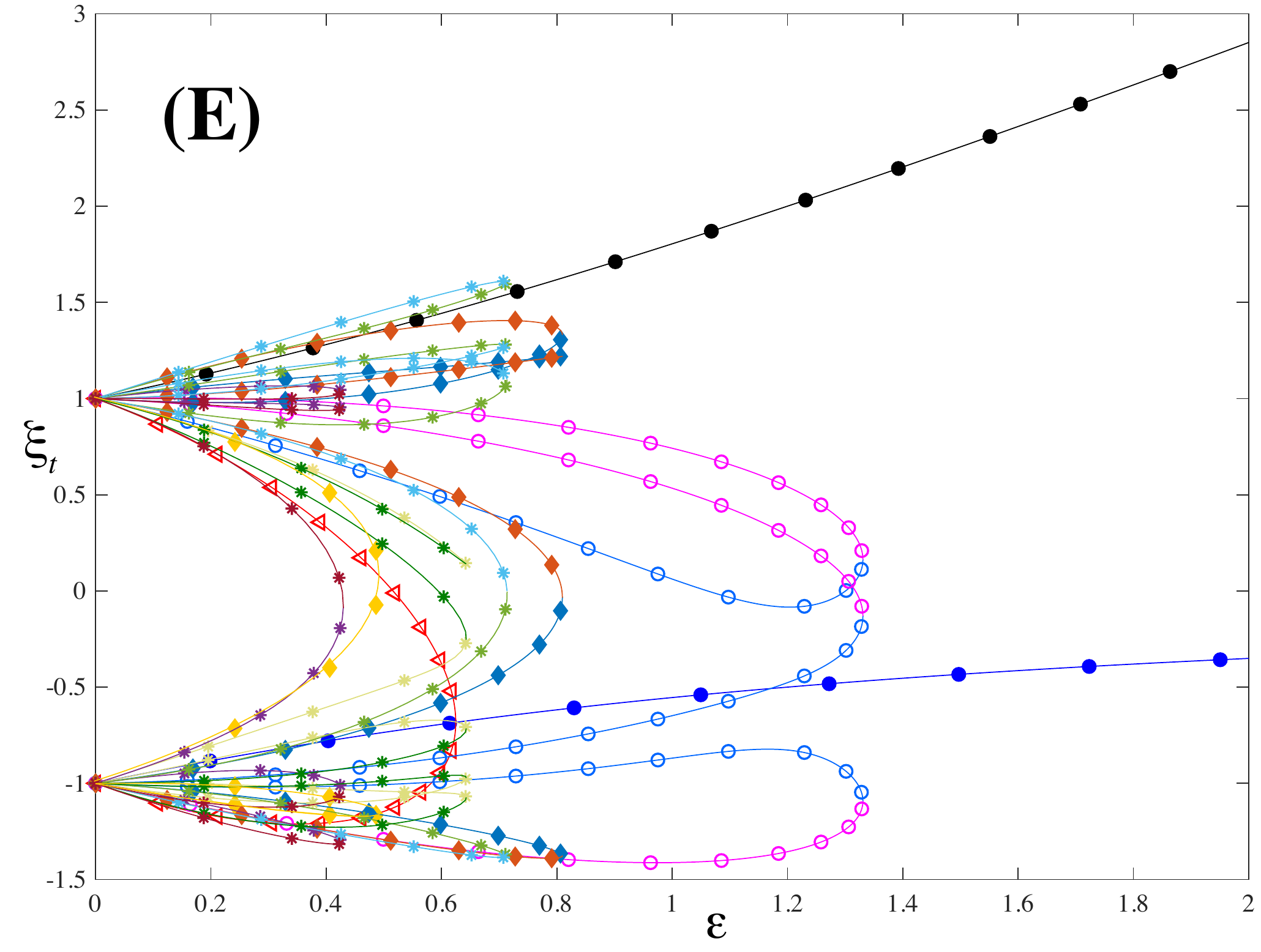}{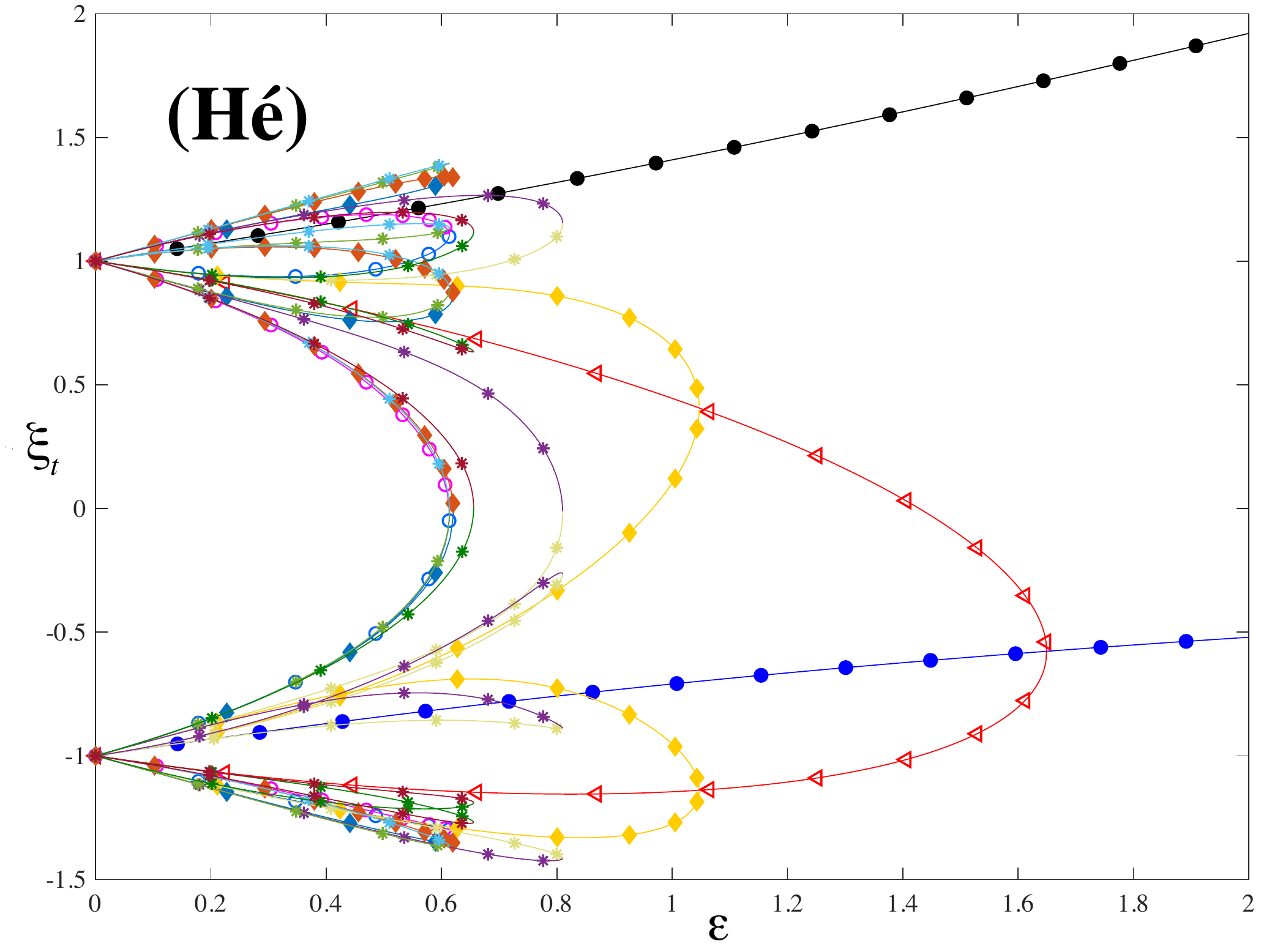}{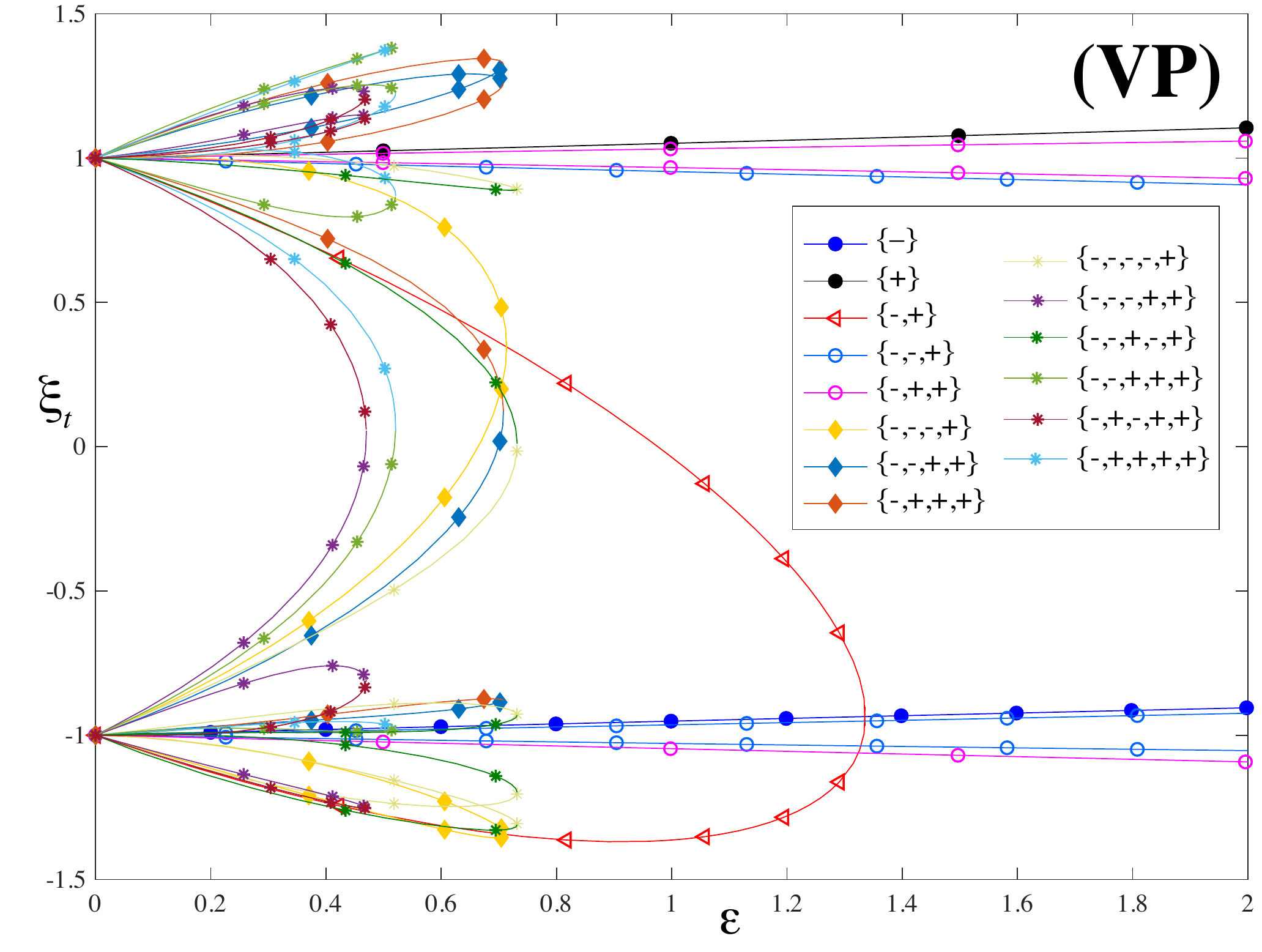}{Points $\xi_t$ on periodic orbits with periods one to five as $\eps$ varies for the four examples of \Tbl{4ExTab}. Each orbit is labeled by its AI symbol sequence.
}{BifurcationDiagrams}{0.5}

Recall that when $b=0$ the AI state corresponds to $\xi_t = s_t$, so all of the orbits in these cases emerge from these two points in \Fig{BifurcationDiagrams}. For the parallel lines case the AI state is nontrivial.
Nevertheless, since this case has slope $m = -b/2a = 0.8 <1$, we iterate the contraction \Eq{MapAsFxn} starting with $\xi = s$
to give the AI state to start the continuation at $\eps = 0$.

Bifurcations for the four cases for all orbits up to period six are summarized in \Tbl{BifTab}.\footnote
    {The $\{-\} \to \{-,+\}$ period doubling occurs at $\eps=\frac{2(a-c)}{\sqrt{(1+\sigma+\delta)(a(3+3\sigma-\delta)+b(1+\sigma+\delta)+c(-1-\sigma+3\delta))}}$. This reproduces the results in the first line of \Tbl{BifTab}.}
Note that all of these bifurcation values lie beyond $\eps_N$, as given in \Tbl{4ExTab}. For the case (H\'e), these results are consistent with bifurcations of the 2D H\'{e}non map. In particular, it was
shown in \cite{Sterling99} that there are no bifurcations below $\eps \approx 0.4999$ (see \App{EmbeddedHenon}), and here, the smallest observed bifurcation for the (H\'{e}) case is a saddle-node of two period-six orbits at $\eps \approx 0.61$.

\begin{table}[ht]
\centering
\begin{tabular}{c|c|c|c|c|c|c}
    Parent & Type & Child & \multicolumn{4}{c}{Case} \\
    &  &  & ($\parallel$) & (E) & (H\'{e}) & (VP) \\
    \hline
    $\{-\}^\infty$ & pd
    & $\{-,+\}^\infty$  & 1.75 & 0.63 & 1.64 & 1.34\\ 
    \hline
     $\{-,+\}^\infty$ & pd & $\{-,-,-,+\}^\infty$  & 1.39 & 0.51  & 1.09 & 0.71\\ 
    \hline
    $\{-,+,+\}^\infty$ & pd & $\Tvec{\{-,-,+,+,-,+\}^\infty}{\{-,-,+,-,+,+\}^\infty}$   
                                & $\Tvec{1.68}{-}$ &  $\Tvec{1.10}{-}$ & 
                                $\Tvec{0.61}{-}$ & $\Tvec{-}{1.40}$\\
    \hline
     & sn & $\{-,\pm,+\}^\infty$   & 1.72 &  1.33 & 0.61 & 10.89\\
    \hline
     & sn & $\{-,\pm,+,+\}^\infty$ & 1.63 & 0.81 & 0.62 & 0.71\\ 
    \hline
     & sn & $\Tvec{\{-,-,\pm,-,+\}^\infty}{\{-,-,-,\pm,+\}^\infty}$       & 
                                $\Tvec{1.27}{-}$ & $\Tvec{0.64}{-}$   & 
                                $\Tvec{-}{0.81}$ & $\Tvec{0.73}{-}$\\ 
    \hline
     & sn & $\Tvec{\{-,\pm,-,+,+\}^\infty}{\{-,+,-,+,\pm\}^\infty}$  & 
                                $\Tvec{1.13}{-}$ & $\Tvec{0.43}{-}$ & 
                                $\Tvec{-}{ 0.66}$ & $\Tvec{0.47}{-}$\\
   \hline
     & sn & $\{-,\pm,+,+,+\}^\infty$  & 1.60 & 0.71 & 0.61 & 0.52\\
    \hline
      & sn & $\{-,-,-,\pm,-,+\}^\infty$ & 1.32  & 0.52 & 0.97 & 0.70\\
     \hline
     & sn & $\Tvec{\{-,\pm,-,+,+,+\}^\infty}{\{-,+,-,+,+,\pm\}^\infty}$ &  
                              $\Tvec{1.10}{-}$ & $\Tvec{0.44}{-}$ & 
                              $\Tvec{-}{0.70}$ & $\Tvec{0.54}{-}$\\
      \hline
     & sn & $\{-,\pm,+,+,+,+\}^\infty$ &  1.59 & 0.72 & 0.61 & 0.56\\
    \hline 
     & sn & $\Thvec{\{-,-,\pm,-,+,+ \}^\infty}{\{-,\pm,-,-,+,+\}^\infty}{\{-,-,-,\pm,+,+\}^\infty}$ &
                              $\Thvec{1.14}{-}{-}$ & $\Thvec{0.62}{-}{-}$ & 
                              $\Thvec{-}{-}{0.83}$ & $\Thvec{-}{0.55}{-}$\\
\end{tabular}
\caption{Parameter, $\eps$, for period-doubling (pd) and saddle-node (sn) bifurcations 
for all orbits up to period six for the four examples of \Tbl{4ExTab}.
Orbits are identified by their symbol sequences in the first and third columns.
For saddle-node bifurcations, the symbol sequences of the two colliding orbits are 
listed together: the $\pm$ indicates the single symbol that differs.
A dash indicates the bifurcation does not occur with that pair for that case.}
\label{tbl:BifTab}
\end{table}

Notice that there appears to be a pattern in the symbol sequences for orbits that bifurcate: the symbol sequences of bifurcating orbits differ in only one symbol. Indeed, we list the two orbits involved in a saddle-node in \Tbl{BifTab} as a single symbol sequence, denoting the sole symbol that differs by $\pm$.  This pattern is also observed for a number of other parameter sets and periodic orbits with periods up to $75$. Since these are the two generic codimension-one bifurcations, we propose the following conjecture.

\begin{con}\label{con:Bifurcation}
Codimension-one bifurcations of periodic orbits continued from the AI limit of the map \Eq{QVPMap} only occur between orbits with exactly one differing symbol.
\end{con}

Note that this conjecture also appears to hold for the \hen map \cite{Sterling99}, but we are unaware of a proof even in this case.

To gain more intuition on the behavior of orbits under continuation, an orbit with period $n=2500$
for an arbitrarily chosen symbol sequence is shown in \Fig{2DLargePeriodExamples} for each of the
four cases in \Tbl{4ExTab}. Each orbit is shown projected onto the $(\xi_{t-1},\xi_t)$-plane for
four values of $\eps$. When $\eps$ is small, the orbits lie close to the curve \Eq{EDelta},
like the AI states shown in \Fig{AILargePeriod}. As $\eps$ grows they move away from the curve and
form clusters and clusters of clusters. In each case, the maximum $\eps$ shown is the largest value
for which the continuation converged. The self-similar behaviour in the clusters becomes more apparent
in \Fig{2DLargePeriodExamplesQ1}, where we enlarge the first quadrant of \Fig{2DLargePeriodExamples}
and only show these orbits for the maximum $\eps$.

\InsertFigFour{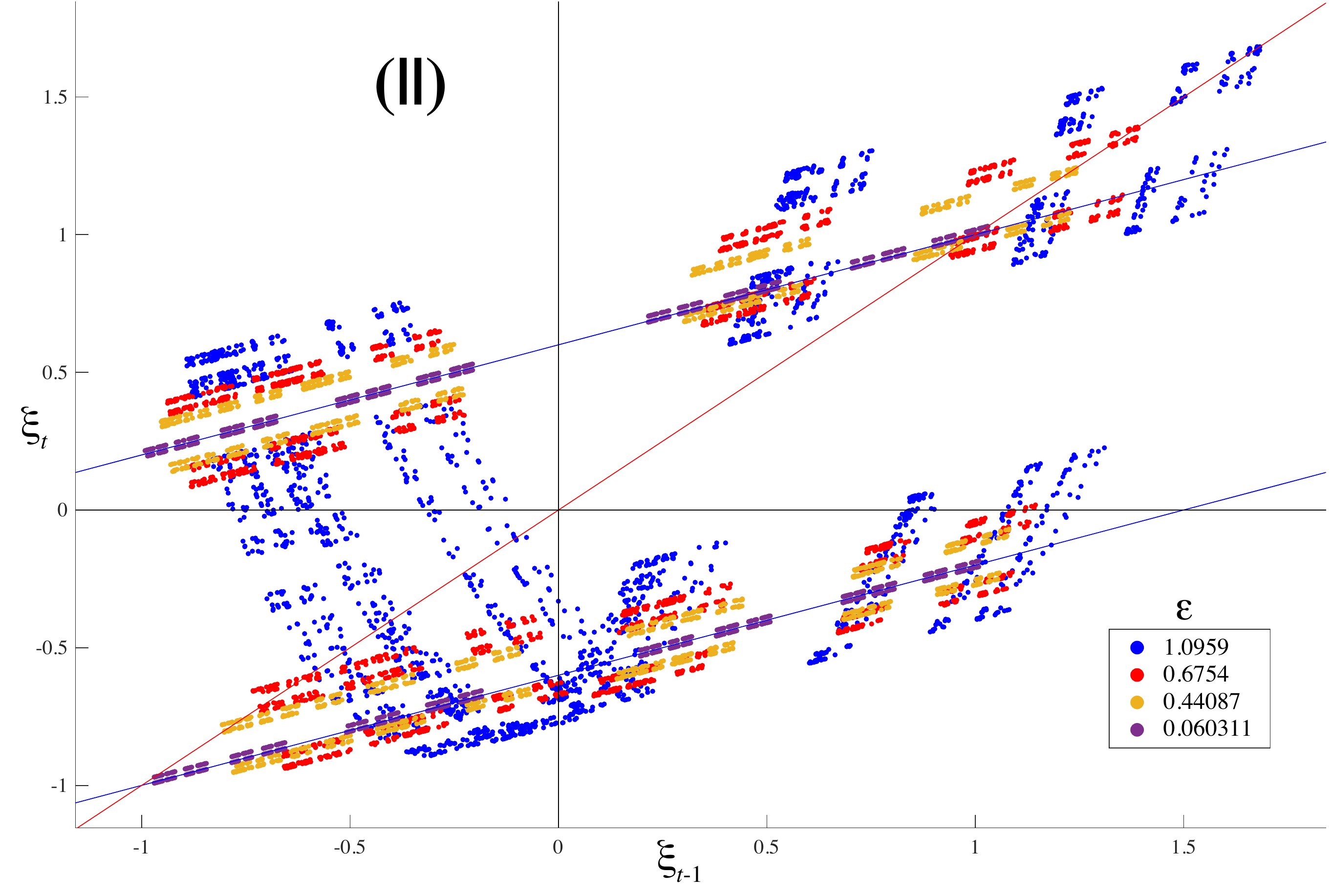}{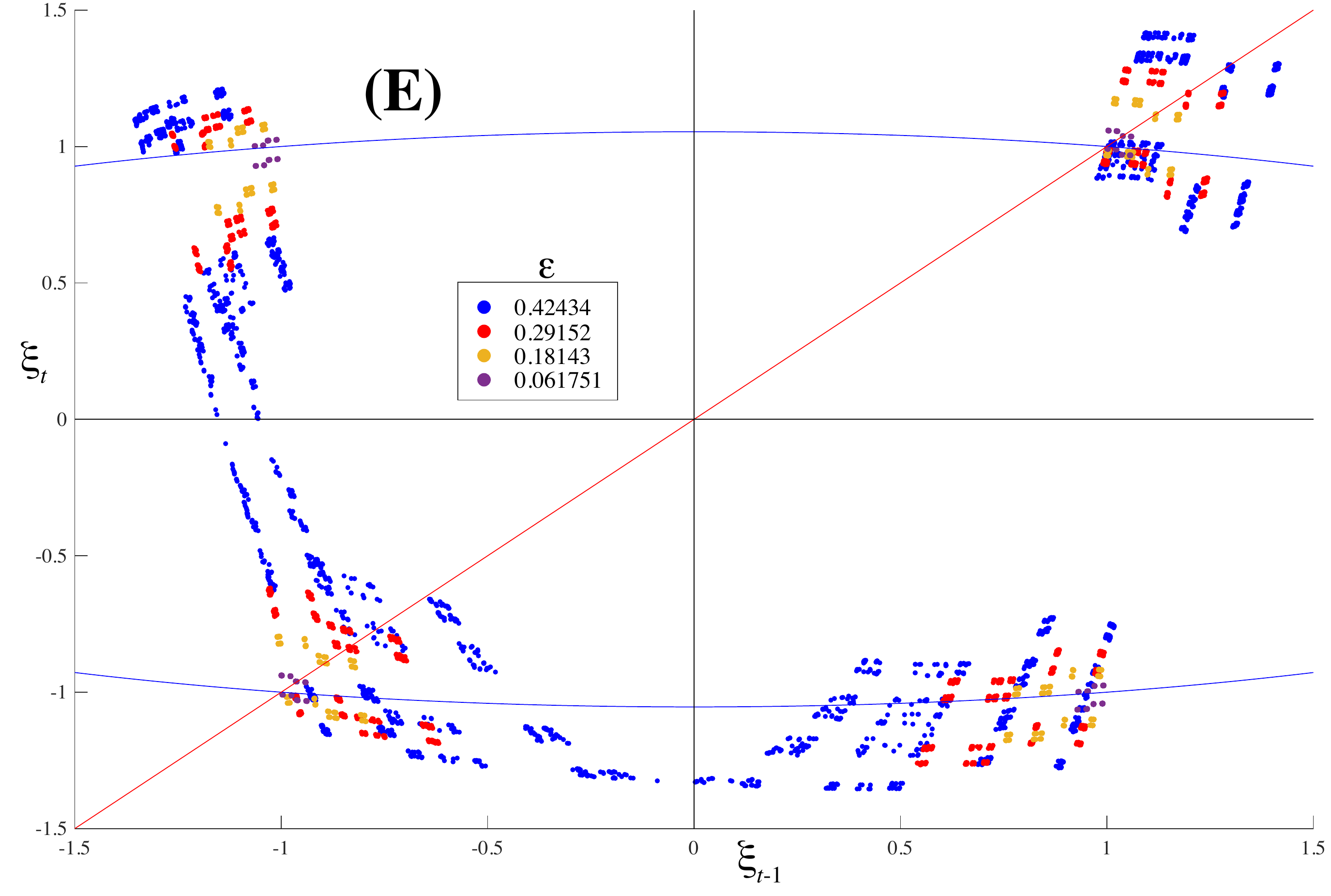}{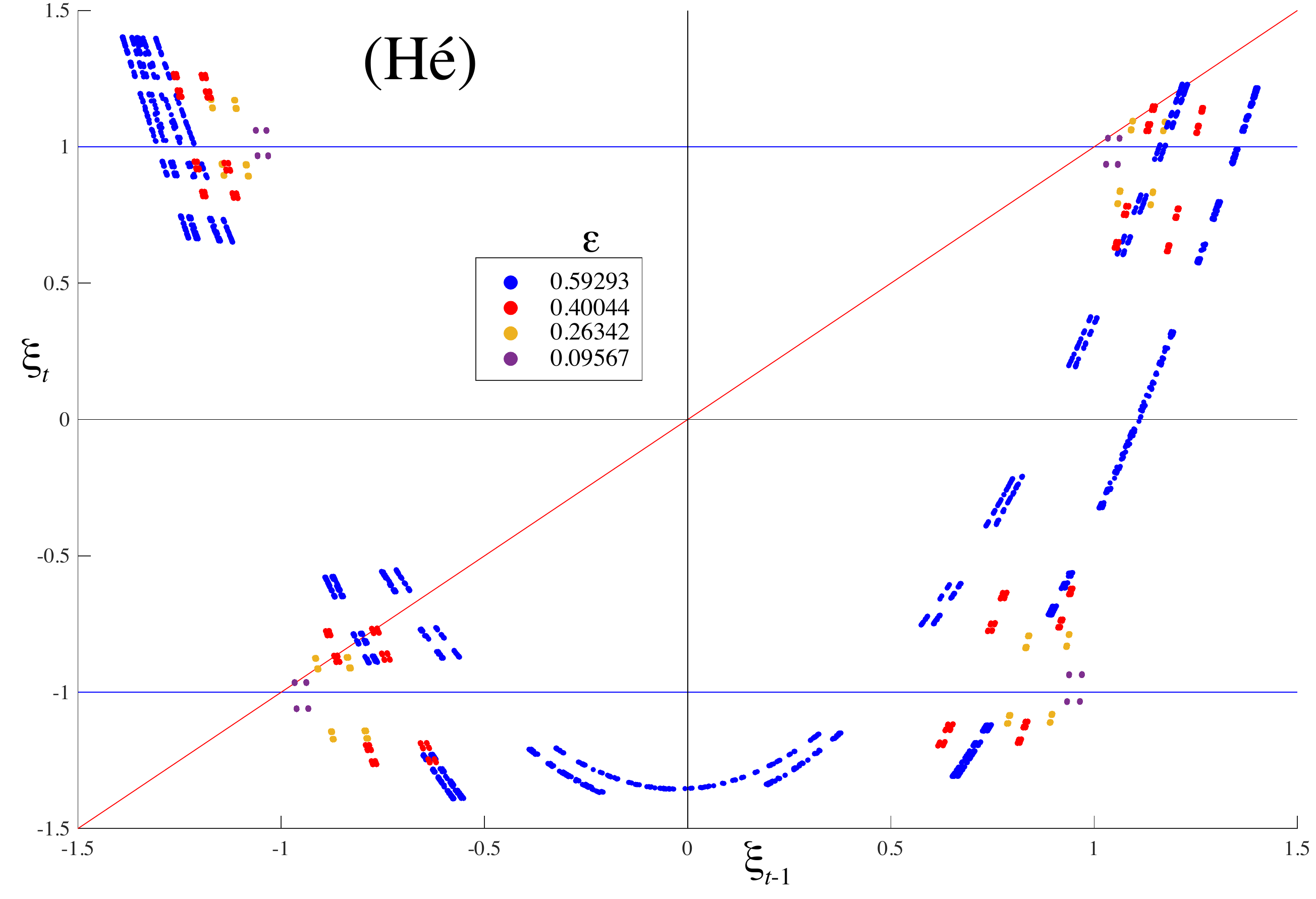}{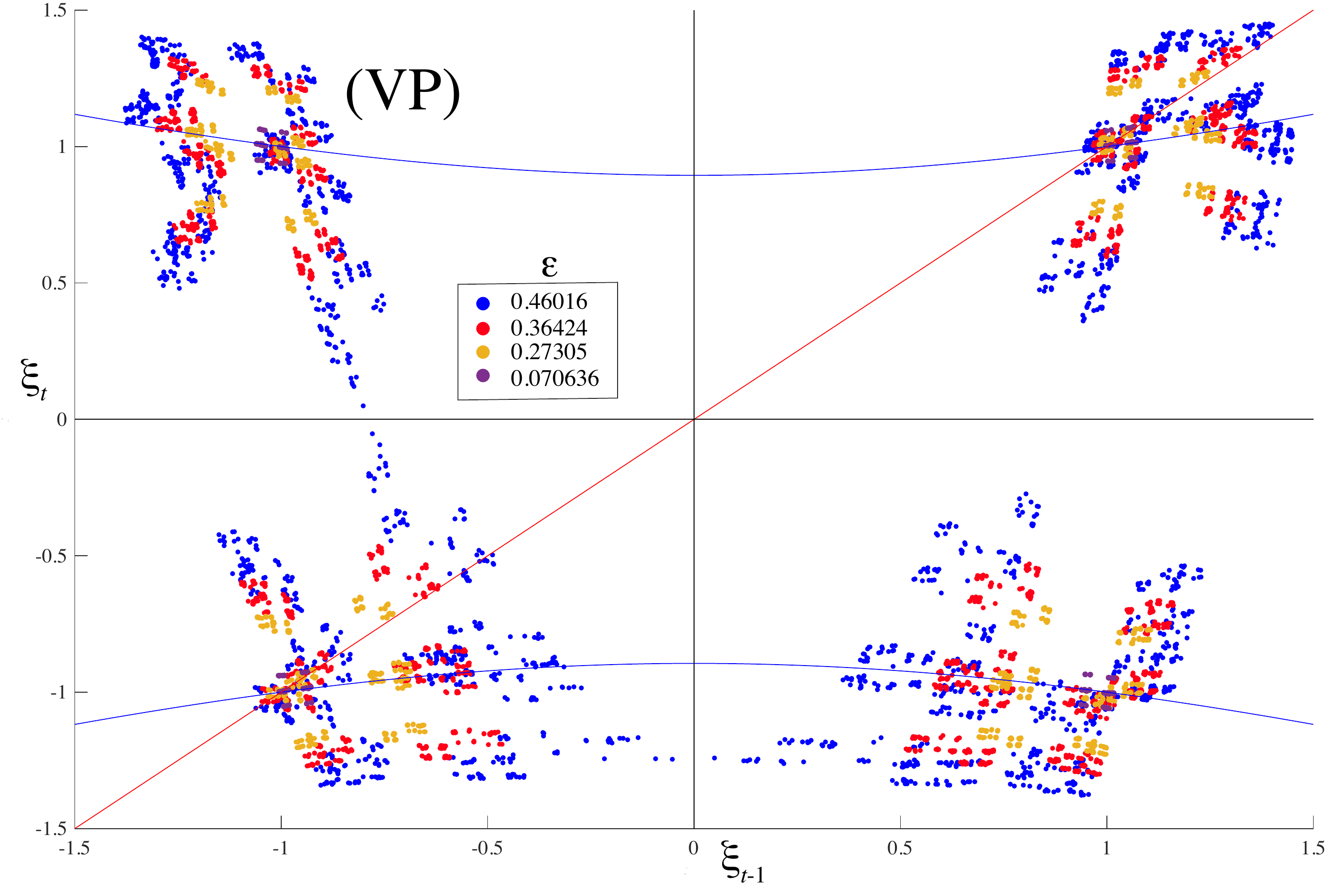}{Continuation of a period-2500 orbit with an arbitrarily chosen symbol sequence for the parameters of \Tbl{4ExTab} and  four values of $\eps$, as shown in the legends. The blue curve is the quadratic $Q(\xi_t,\xi_{t-1})=1$ and the red line is the diagonal $\xi_t=\xi_{t-1}$.}
{2DLargePeriodExamples}{0.5}

Finally, \Fig{3DLargePeriodExamples} shows the three-dimensional structure of the orbit in \Fig{2DLargePeriodExamples} for each of the cases of \Tbl{4ExTab}. Also shown in the figure are the two quadratic surfaces $Q(\xi_t,\xi_{t-1})=1$ and $Q(\xi_{t+1},\xi_t)=1$. Recall that in 3D, AI states must fall on  the intersection of these surfaces,  as illustrated in \Fig{3DAISurface}. For the first three cases, the structure in three-dimensions resembles the standard folded horseshoe shape of the H\'enon attractor; indeed all three of these maps are volume-contracting, with $\delta < 1$, and of course, the (H\'e) case is indeed a 2D H\'enon map embedded in 3D. The fourth case, which is volume-preserving, exhibits a much more complex three-dimensional structure, though it still appears to be related to a fold in the $(\xi_{t+1}, \xi_{t})$ plane, which
is equivalent to the $(x,y)$ plane in \Eq{QVPMap}.


\InsertFigFour{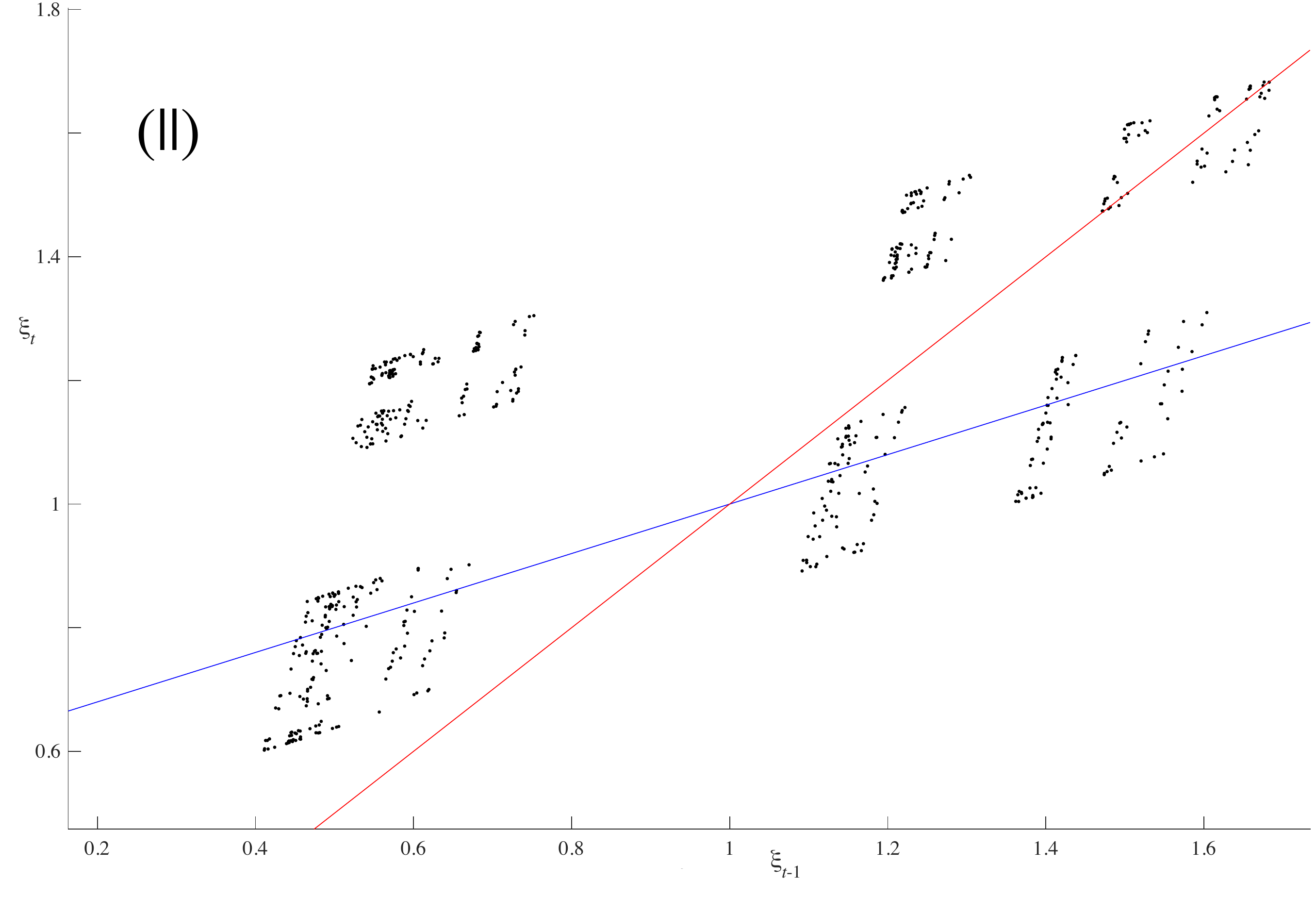}{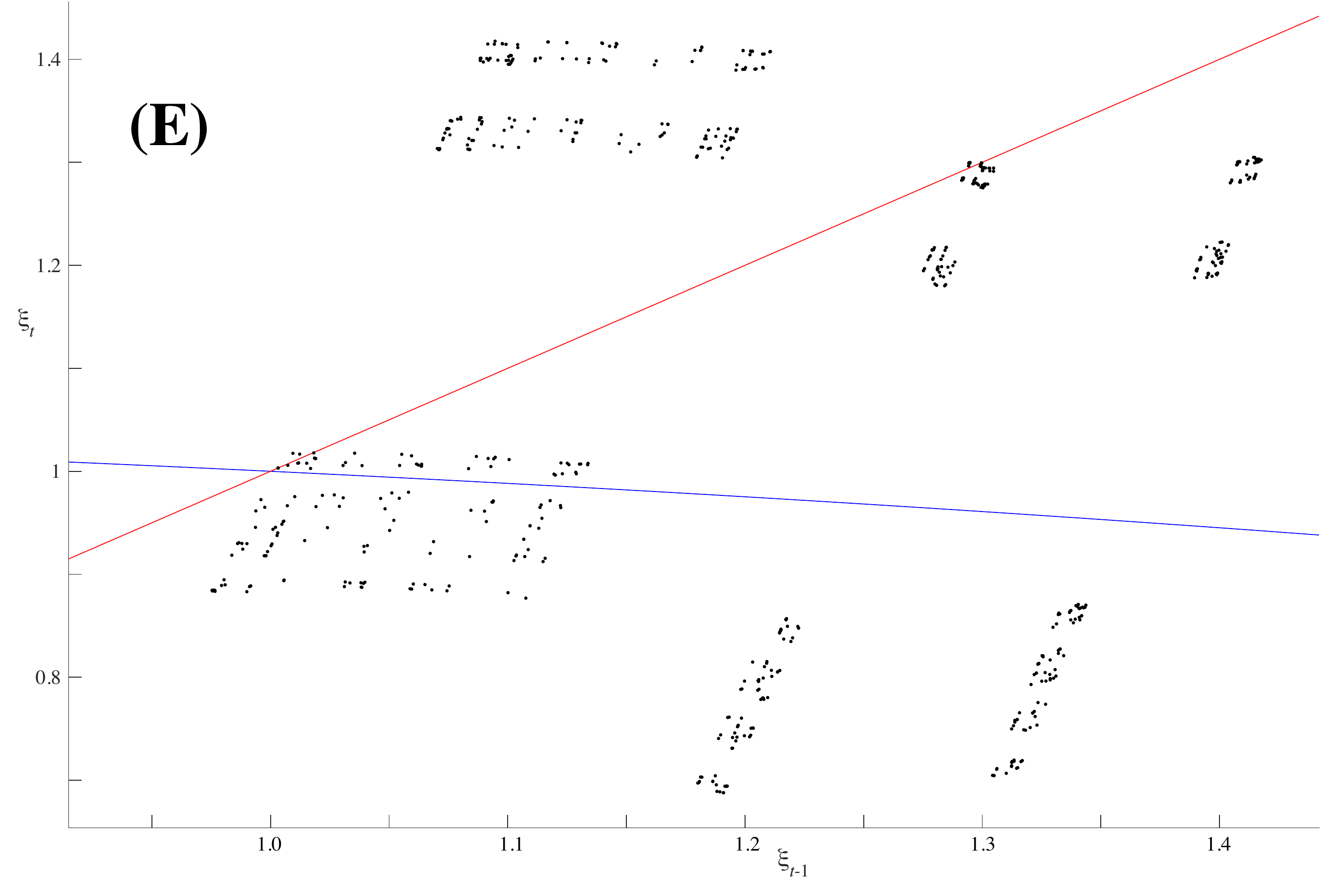}{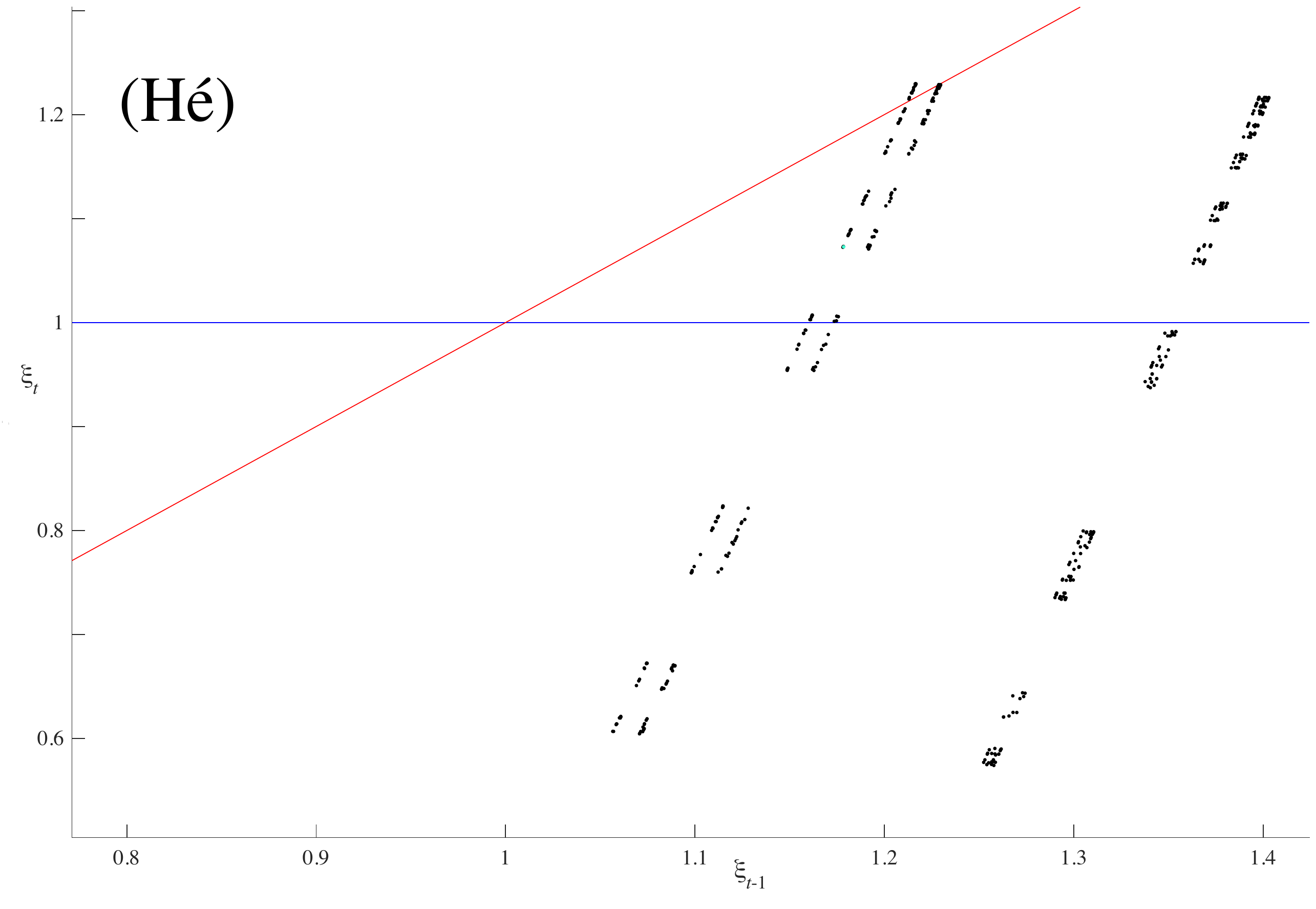}{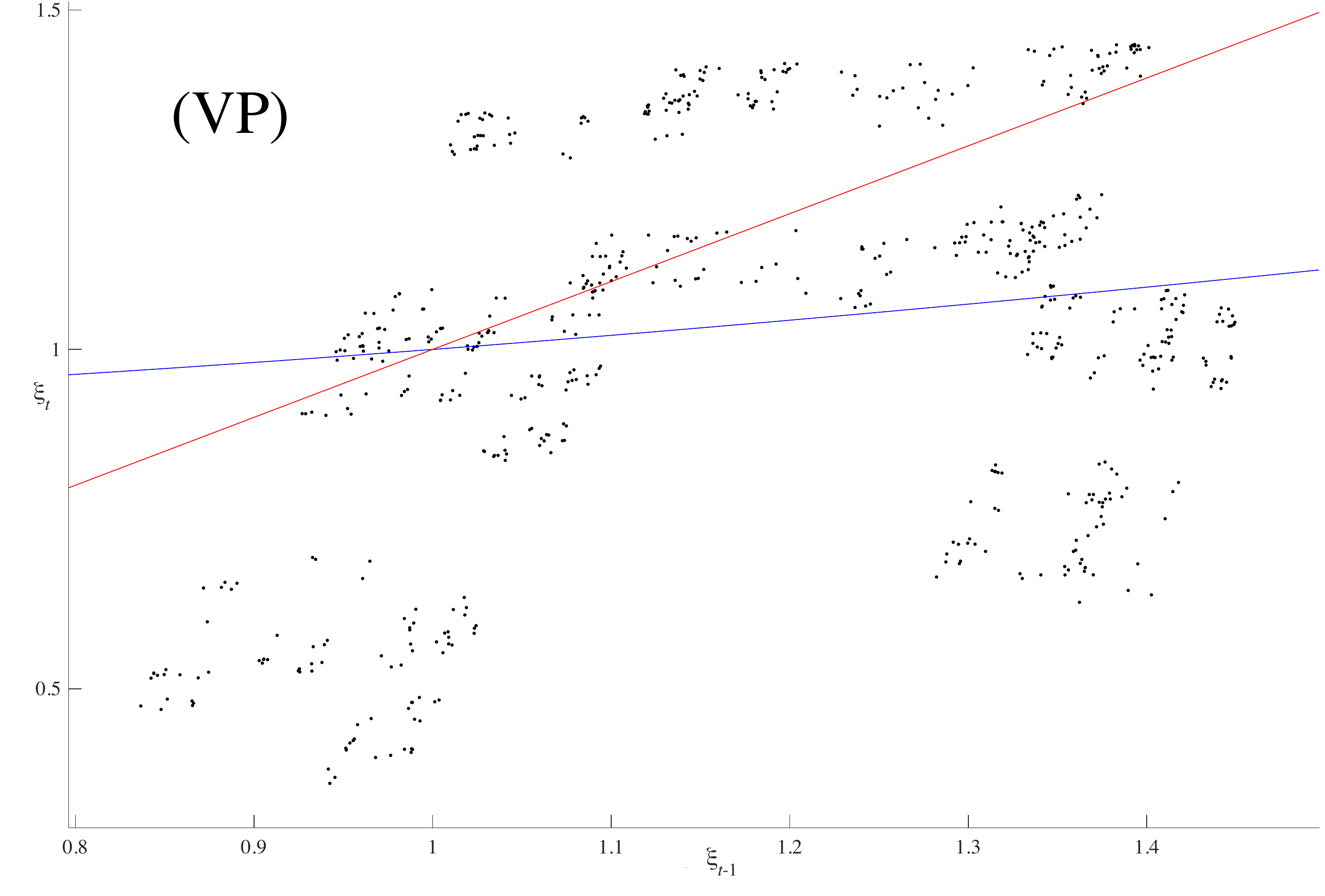}{An enlargement of the period-2500 orbit for the largest $\eps$ of \Fig{2DLargePeriodExamples} near the point (1,1) showing approximately self-similar structures. The blue curve is $Q(\xi_t,\xi_{t-1})=1$ and the red line is the diagonal $\xi_t=\xi_{t-1}$.}
{2DLargePeriodExamplesQ1}{0.5}

\InsertFigFour{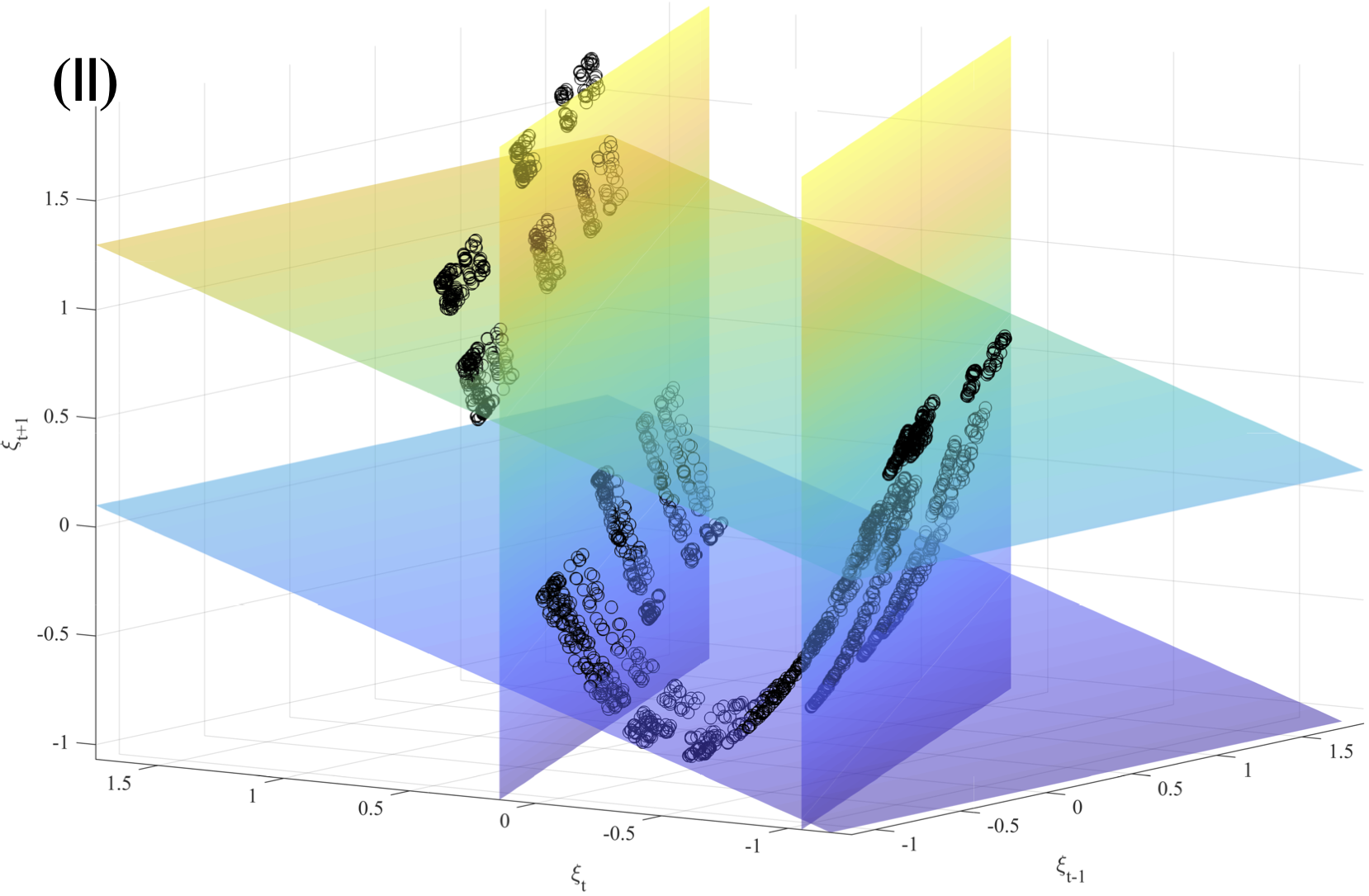}{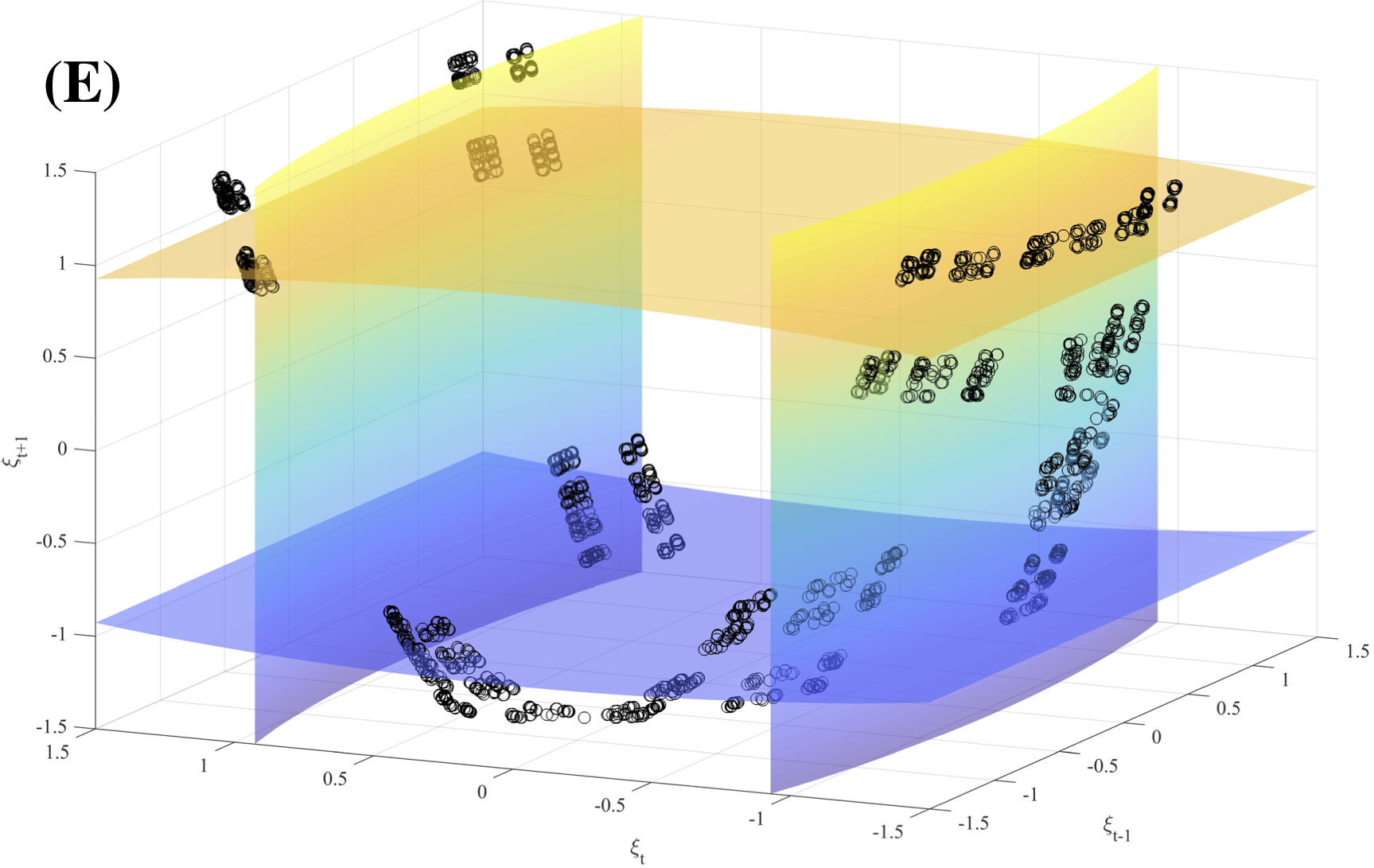}{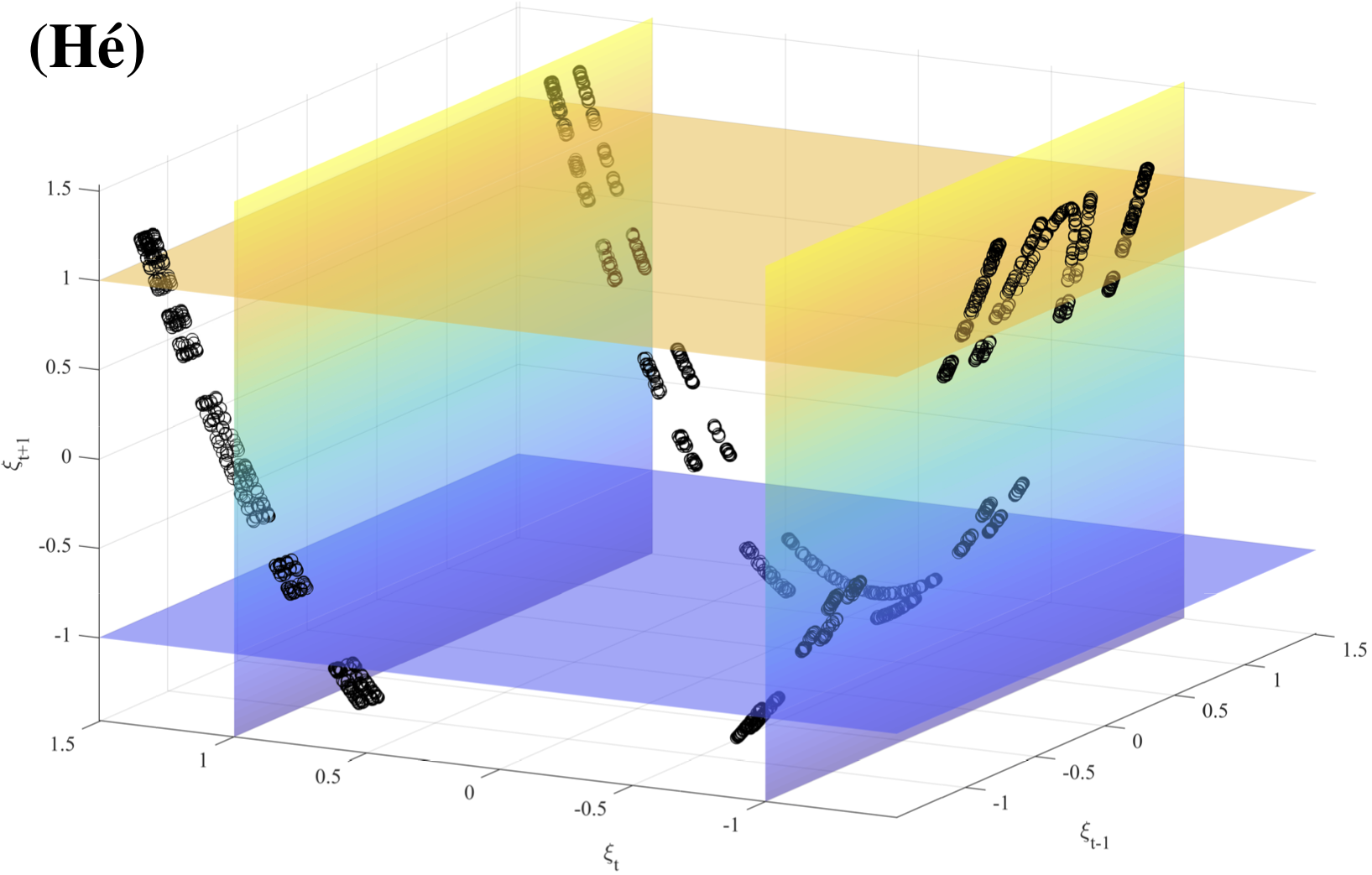}{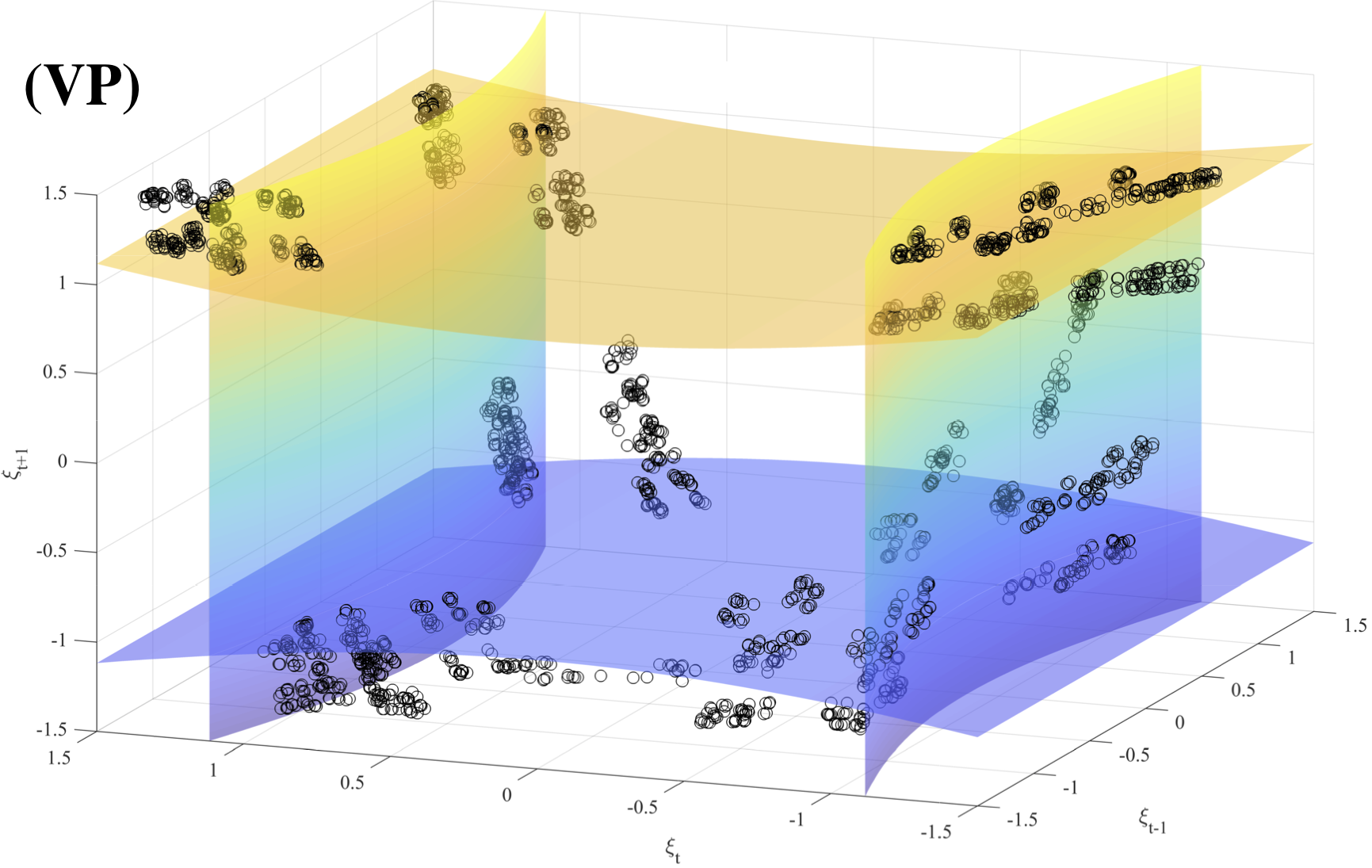}{The period-2500 orbit of \Fig{2DLargePeriodExamples} of the four examples of \Tbl{4ExTab} in $\bR^3$ for ($\parallel$) $\eps\approx1.0959$, (E) $\eps\approx0.42434$, (H\'{e}) $\eps\approx0.59293$, and (VP) $\eps\approx0.46016$. Also shown are the surfaces 
$Q(\xi_t, \xi_{t-1}) = 1$ and $Q(\xi_{t+1},\xi_t)= 1$}
{3DLargePeriodExamples}{0.5}

\section{Conclusions}\label{sec:Conclusions}

Previous research on anti-integrable limits for quadratic 3D maps has been limited to the case that
the AI limit corresponds to either a deterministic 1D map or to a full shift on two symbols.
We have argued that the AI limit more generally corresponds to a quadratic correspondence, like \Eq{AIRelation}.
Now the symbolic dynamics arises from the choice of branch at each iteration. In previous studies, either only one branch of the correspondence was used, or it was simplified by assuming AI states are identical to their symbols, e.g., the case $b=c=0$.
In our study, the AI states are still nondeterministic---as in the original paper on anti-integrability \cite{Aubry90}---but now each symbol sequence determines a sequence of one-dimensional maps whose orbits correspond to the AI states. For the cases that we study, this 1D map sequence is a contraction for \textit{any} given symbol sequence; thus, just as in \cite{Aubry90}, 
we still get a collection of AI states that has the dynamics of a full shift on two symbols.

Here we limited our study to only one of the possible AI limits listed in \Sec{3DAI}, namely $\alpha \to -\infty$.
This case, however, does have the main feature of the AI limit for cases (2) and (3): a quadratic correspondence. In a follow-up paper, we plan to generalize these results to include other cases. For example, it would be interesting to understand how the various AI limits connect as the parameters vary from one sector at $\infty$ in the $(\alpha,\sigma)$ plane to another.  It would also be interesting to investigate AI limits that allow the coefficients $a,b,c$ of $Q$ to approach infinity as well.

Both the analysis at the AI limit and for continuation away from the limit took advantage of the contraction mapping theorem to prove the existence and uniqueness of orbits for each symbol sequence. In \Sec{Existence}, we studied the three types of quadratic curves and found a region $\cR$ of parameters for which the map $\cF$ \Eq{AIMapF} is a contraction on $\cB$. In \Sec{PersistOrbit} we limited the study to two fairly simple cases, where one can easily show the map $T(\xi;s)$ \Eq{TMap} with parameters $(a,c,\gamma,M)\in\cR_M$ remains a contraction for some $\eps > 0$ on a domain $\cB_M$, recall \Eq{B^b0_M} and \Eq{B_M^||}, for a bound $M$. Nevertheless, numerical results from a standard predictor-corrector continuation algorithm shown in \Sec{Continuation} indicate that our analysis could be extended to larger parameter and phase space regions. 

For the 2D, area-preserving H\'enon map, \cite{Dullin05} found that the rotational periodic orbits,those that encircle the elliptic fixed point, are the continuation of AI states with symbol sequences of a particular ``rotational'' class. It is interesting to speculate that there could be a similar correlation between special symbol sequences at the AI limit and rotational orbits in the 3D volume-preserving case. For example, \cite{Dullin08a} found parameter regimes in which the map has an invariant circle surrounded by a (Cantor) family of nested invariant tori. 

The continuation algorithm in \Sec{Continuation} showed that orbits can be followed from the AI limit beyond the theorem-limited results for a maximum $\eps$-value. This suggests that our analysis can indeed be used to better understand structural development of orbits as parameters vary. We propose that these ideas could be used to prove the existence of cantori and understand the resulting formation of tori, i.e., to understand how chaotic structures undergo bifurcations to regularity. 

An example of such a regular structure is shown in \Fig{ClosedCurve}, projected onto the $(\xi_{t-1},\xi_t)$-plane
for $(a,c,\sigma,\delta)=(0.5,0.5,-0.3,0.5)$.
For these parameters, the map \Eq{QVPMap} appears to have an attracting invariant
circle---the set of blue points in \Fig{ClosedCurve}---when $\alpha=-1$ (or equivalently when $\eps=1$).
Can we find this orbit using continuation from the AI limit?
When $\eps = 0$, these parameters lie outside of $\cR$, so \Lem{AIContraction} does not guarantee
a unique orbit for each symbol sequence.
Indeed, since $a=c$ and $b=0$ the curve \Eq{EDelta} is a circle centered on the origin. So while the 
correspondence $Q(\xi_t,\xi_{t-1}) = 1$ is well-defined for any symbol sequence on the domain 
$|\xi| < \sqrt{2}$, it is \textit{not} a contraction and therefore each symbol sequence need 
\textit{not} correspond to a unique orbit. To obtain an approximately periodic orbit that 
reproduces that seen in \Fig{ClosedCurve}, we iterate the point 
$(x,y,z) =(-0.7222,-1.4040,0.1022)$ 
on the invariant circle until it 
returns within a distance of $0.005$ of itself; this first happens at $t = 1089$. 
We then find a symbol sequence by solving \Eq{TMap} for $s_t$ to give a sequence 
$\{s_0,s_1,\ldots, s_{1088}\}$ associated with this nearly periodic orbit. 
This is taken to represent a period-1089 AI state with $\xi_t = s_t$ (since $b = 0$), at $\eps = 0$.
The method discussed in \Sec{Continuation} continues this to a periodic orbit of \Eq{QVPMap} up 
to $\eps\approx 0.2721$; the resulting orbit (now in the scaled $\xi$ coordinates) is shown as the black points in \Fig{ClosedCurve}. 

\InsertFig{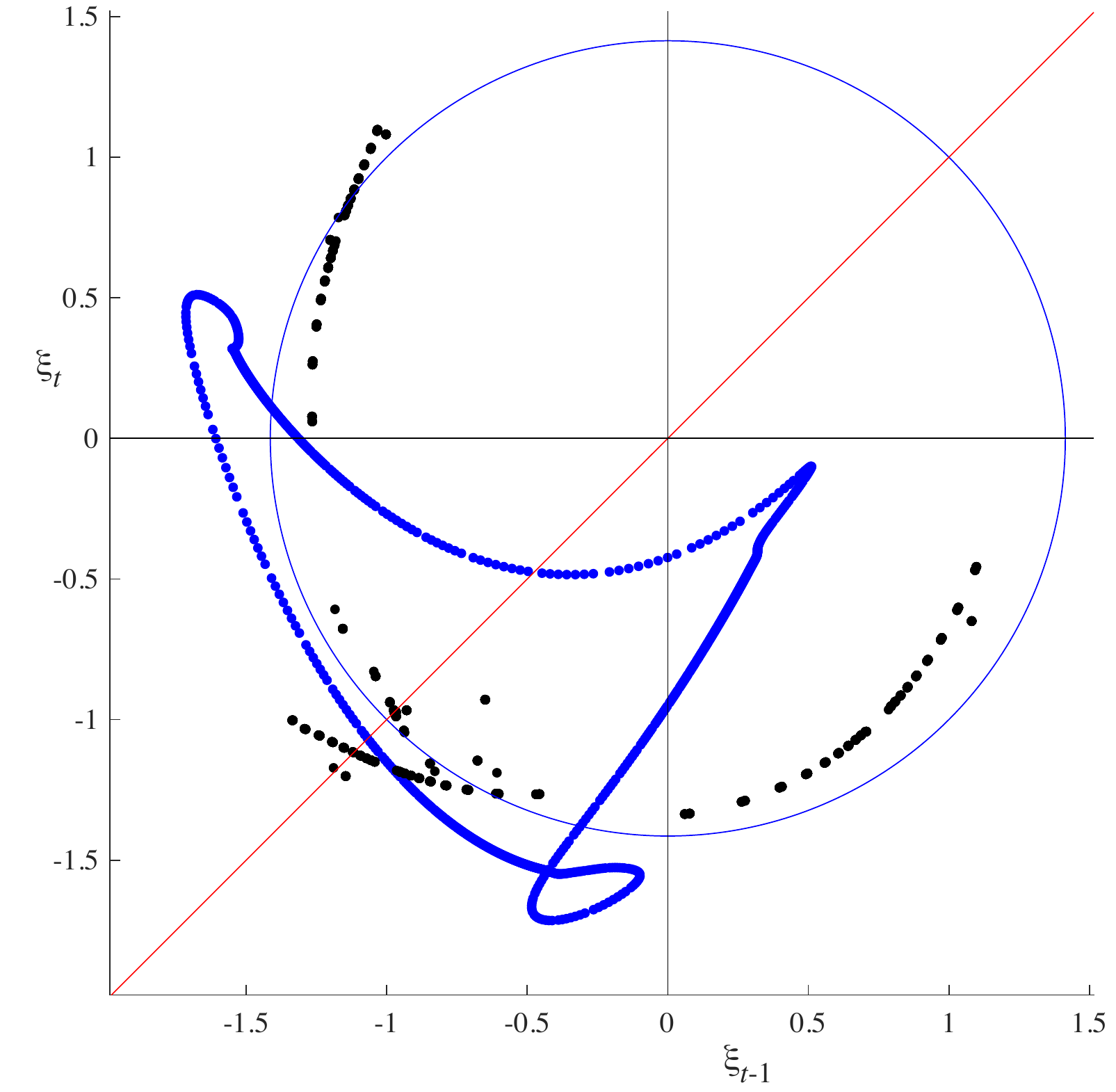}{Two-dimensional projection of an attracting invariant circle (blue points) of the map \Eq{QVPMap} with parameters $(a,c,\sigma,\delta)=(0.5,0.5,-0.3,0.5)$ at $\alpha=-1$ ($\eps=1$). The black points show a period $1089$ orbit with symbol sequence obtained by solving \Eq{TMap} for $s_t$ using points on the invariant circle. This sequence is continued from $\eps = 0$ to $\eps\approx0.2721$, the maximum $\eps$ for which the continuation algorithm converged. Also pictured are the quadratic curve \Eq{EDelta} (blue), and the diagonal $\xi_{t}=\xi_{t-1}$ (red).}{ClosedCurve}{0.5}

Although this orbit has not (yet) reached the attracting loop at $\eps = 1$, it is interesting to note that tests of twenty other randomly generated period-$1089$ symbol sequences gave orbits that continued at most to $\eps \approx 0.03$, suggesting that the sequence derived from the invariant circle is unusually robust. This promising---if incomplete---result  shows the potential of our approach of using anti-integrability to understand more about the development of regular structures as parameters are varied.


\newpage
\appendix
\section{Appendix}

%
\subsection{Elliptic Contraction}\label{app:EllipticContractionArguments}

To find the domain, $D$ where the elliptic curve $Q(\xi_t,\xi_{t-1})=1$ has absolute slope less than one, differentiate \Eq{MapAsFxn} to obtain
\beq{MapAsFxnDeriv}
    f_{s}'(\xi_{t-1}) = -\frac{b}{2a}+ \frac{s\Delta\xi_{t-1}}{2a\sqrt{4a+\Delta \xi_{t-1}^2}} .
\eeq
This implies that $|f'_s|<1$ in an interval for each branch, so that $D$ is the intersection of these intervals:
\begin{align*}
    D  &=  \frac{1}{\sqrt{|\Delta|}}\left(-a+c-1,\frac{3a+c-1}{\sqrt{2a+2c-1}}\right) \cap     \frac{1}{\sqrt{|\Delta|}}\left(\frac{-3a-c+1}{\sqrt{2a+2c-1}},a-c+1\right)\\
    &=\left( -\Pi,\Pi \right) \text{ for } 
    \Pi=\frac{1}{\sqrt{|\Delta|}}\min \left(|a-c+1| ,\frac{|3a+c-1|}{\sqrt{2a+2c-1}}\right) .
\end{align*}
For the orbit to exist at the AI limit, the range \Eq{xcEllipse} must be be a subset of $D$, which implies
\[
    2\sqrt{c} \leq \min \left(|a-c+1| ,\frac{|3a+c-1|}{\sqrt{2a+2c-1}}\right).
\]
Solving for $c$ and noting the two terms in the $\min$ are equal at $a=\frac{-1+\sqrt{5}}{2}$, we obtain \Eq{EllipticContraction}.

\subsection{Persistence for Vanishing \textit{b}}\label{app:VanBPersistenceArgument}

Using \Eq{B^b0_M} and $\gamma$ as defined in \Eq{gamma}, we find the conditions necessary to guarantee $T^{0}(\xi;s)$, \Eq{VanBOp}, is a contraction.
We first bound $T^{0}$ from above
\begin{align*}
         \| T^{0}(\xi;s) \|_\infty &= \left\| s_t \sqrt{\frac{1}{a}\left[1 + \eps (\xi_{t+1}+\sigma \xi_{t-1}-\delta \xi_{t-2})-c\xi_{t-1}^2\right]} \right\|_\infty\\
         &\leq \sqrt{\frac{1}{|a|}(1 + \gamma \| \xi \|_\infty  + |c| \| \xi \|_\infty^2)}\\
         &\leq \sqrt{\frac{1}{|a|}\left[1+ \gamma (1+M)  + |c| (1+M)^2\right]} .
    \end{align*}    
Requiring this expression to be no more than $1+M$ gives
\beq{bVanishes1.2}
  \gamma \leq \frac{(|a|-|c|)(1+M)^2-1}{1+M} .
\eeq
Recalling that $c=1-a$  and taking $a>0$ gives the first expression in \Eq{bVanishes1}. Similarly, bounding the operator from below gives
 \begin{align*}
        \| T^{0}(\xi;s)\|_\infty &\geq  \sqrt{\frac{1}{|a|}(1-\gamma \| \xi \|_\infty - |c| \| \xi \|_\infty^2)} \\
        & \ge \sqrt{\frac{1}{|a|}\left[1-\gamma (1+M) - |c| (1+M)^2 \right]} .
\end{align*} 
Requiring this expression be larger than $1-M$ then gives
\beq{bVanishes1.2new}
        \gamma  \leq \frac{1- |a|(1-M)^2- |c| (1+M)^2}{1+M} .
\eeq
Thus the conditions \Eq{bVanishes1} are required for $T^{0}: \cB_M \to \cB_M$.

The derivative of $T^{0}$ is the matrix
\[
    D_{j}T^{0}_t(\xi;s) = \frac{s_t}{2 a T^{0}_t(\xi;s)} \left[\eps(\Delta_{j,t+1} +\sigma \Delta_{j,t-1} - \delta \Delta_{j,t-2})
                -2c \xi_{t-1} \Delta_{j,t-1} \right] ,
\]
where $\Delta_{i,j}$ is the Kronecker-delta.
When $\xi$ and $T^{0}(\xi;s) \in \cB_M$, this gives the bound
\[
    \|DT^{0}(\xi;s)\|_{\infty}
         \le \frac{\gamma + 2|c| (1+M)}{ 2|a|(1-M)} .
\]
Forcing this value be strictly less than one and substituting $c=1-a$ gives the condition \Eq{bVanishes2}.


\subsection{Persistence for Parallel Lines}\label{app:GenParallelLinesPersistence}

To determine the conditions so that $T^\parallel(\xi;s)$, \Eq{ParallelLinesOperator}, is a contraction on \Eq{B_M^||} we first bound
\begin{align*}
    \|T^{\parallel}_t(\xi;s)\|_\infty &=\|m\xi_{t-1}+s_t(1-m)\sqrt{1+\eps(\xi_{t+1}+\sigma \xi_t - \delta \xi_{t-2})}\|_\infty\\
    &\leq|m| \|\xi\|_\infty+|1-m|\sqrt{1+\gamma \|\xi\|_\infty}\\
    &\leq |m|(x^*+M)+|1-m|\sqrt{1+\gamma (x^*+M)} ,
\end{align*}
where $\gamma$ is given in \Eq{gamma} and $x^*$ by \Eq{xStar}.
Requiring that this norm is no more than $x^*+M$ gives
\[
     \gamma(x^*+M) \leq \frac{(1-|m|)^2}{(1-m)^2}(x^*+M)^2 -1 ,
\]
resulting in the condition given in \Eq{ParallelLines1}.

Computing the derivative gives 
\[
    D_{j}T^{\parallel}(\xi;s)_t = m\Delta_{j,t-1} +s_t\eps(1-m) \frac{\Delta_{j,t+1} +\sigma \Delta_{j,t-1} - \delta \Delta_{j,t-2}}{2 \sqrt{1+\eps(\xi_{t+1}+\sigma\xi_{t-1}-\delta\xi_{t-2})}},
\]
which yields
\[
    \|DT^{\parallel}_t(\xi;s)\|_\infty 
    \leq |m| + \frac{\gamma|1-m|}{2\sqrt{1-\gamma(x^*+M)}}.
\]
Requiring $\gamma(x^*+M)<1$ for the radical to remain real and bounding so this value is strictly less than one gives the condition \Eq{ParallelLines2}.

\subsection{Embedded \hen map}\label{app:EmbeddedHenon}

When $\delta = 0$ in \Eq{QVPMap}, the dynamics of the 3D map becomes essentially two-dimensional, because the system is a semi-direct product of the map $(x,y) \to (G(x,y),x)$ with the map $z \to y$.
The $(x,y)$ map has Jacobian $\sigma- bx-2cy$, which generally vanishes along a line, so the 2D map is not globally invertible. For the Jacobian to be everywhere nonzero, $b$ and $c$ must necessarily be zero. Since this results in a map
that is quadratic and invertible, it must then be conjugate to the \hen map \Eq{HenonMap}. Under the standard scaling \Eq{ParamSpace}, $(a,b,c) = (1,0,0)$ and the effective 2D map becomes
\[
    (x',y') = (\alpha -\sigma y + x^2, x) .
\]
This is indeed conjugate to \Eq{HenonMap} under the transformation $(\xi,\eta) = (x,-\sigma y)$
where $k = -\alpha$ and $\delta_H = \sigma$.

The famous strange attractor discovered by \hen \cite{Henon76} occurs when $(k,\delta_H) = (1.4,-0.3)$ where the map is contracting and orientation reversing. To replicate this, we choose in \Tbl{4ExTab} $\sigma = -0.3$ for the case (H\'e), giving an ``embedded \hen map''.
 Using the AI scaling $\alpha=-\eps^{-2}$, the standard \hen attractor will occur when  $\eps=\sqrt{1/1.4}\approx0.845$.
 
 More generally, it was proven in \cite{Sterling99} that there are no bifurcations for the \hen map for
\[
    \eps (1+2|\delta_H|) < 2\sqrt{1-2/\sqrt{5}} \approx 0.6498, 
\]
a region first found by Devaney and Nitecki \cite{Devaney79}. The implication is that every orbit at the AI limit has a unique continuation in this range and there is a hyperbolic horseshoe conjugate to a full shift on two symbols.
For the (H\'e) case of \Tbl{4ExTab} this implies no bifurcations until $\eps \approx 0.4061$, which agrees with observations in \Fig{BifurcationDiagrams} and \Tbl{BifTab}.

\bibliographystyle{alpha}
\bibliography{AIBibliography}

\end{document}